\pdfoutput=1
\documentclass{article}

\usepackage{fullpage}
\usepackage{layout}
\usepackage{multirow}
\usepackage{breakcites}
\usepackage{xr}
\usepackage{authblk}
\usepackage{subcaption}
\usepackage{mathtools}
\usepackage{microtype}
\usepackage[USenglish]{babel}
\usepackage{bbm}
\usepackage{makecell}

\usepackage{amsfonts}
\usepackage{amsmath}
\usepackage{IEEEtrantools}
 \usepackage{amsthm}
\usepackage{amssymb} 

\usepackage{algorithm}
\usepackage[noend]{algpseudocode}

\usepackage{xcolor}
\usepackage{cleveref}
\usepackage{color}
\usepackage{graphicx}
\usepackage{graphics}
\usepackage{tikz,tikz-3dplot}    
\usepackage{tkz-euclide}
\usetikzlibrary{calc}
\usetikzlibrary{decorations.pathreplacing}

\title{Adaptive Sketch-and-Project Methods for Solving Linear Systems}

\date{\vspace{-5ex}}

\author[1]{Robert M. Gower}
\author[2]{Denali Molitor}
\author[2]{Jacob Moorman}
\author[2]{Deanna Needell}
\affil[1]{T\'el\'ecom ParisTech, LTCI, Universit\'e Paris-Saclay, France}
\affil[2]{Department of Mathematics, University of California at Los Angeles, Los Angeles, California, USA}

\newcommand{\R}{\mathbb{R}} 
\newcommand{\N}{\mathbb{N}} 

\newcommand{\cD}{{\cal D}}

\newcommand{\cW}{{\cal W}}

\newcommand{\mA}{{\bf A}}
\newcommand{\mB}{{\bf B}}
\newcommand{\mC}{{\bf C}}
\newcommand{\mD}{{\bf D}}

\newcommand{\mG}{{\bf G}}
\newcommand{\mH}{{\bf H}}

\newcommand{\mI}{{\bf I}}

\newcommand{\mM}{{\bf M}}

\newcommand{\mP}{{\bf P}}

\newcommand{\mR}{{\bf R}}
\newcommand{\mS}{{\bf S}}

\newcommand{\mW}{{\bf W}}

\newcommand{\mZ}{{\bf Z}}

\newcommand{\thup}{^{\text{th}}}

\newcommand{\eqdef}{\overset{\text{def}}{=}} 
\newcommand{\ve}[2]{\langle #1 ,  #2 \rangle} 
\newcommand{\dotprod}[1]{\left< #1\right>} 
\newcommand{\norm}[1]{\left\lVert#1\right\rVert}      

\newcommand{\Prob}[1]{\mathbb{P}[#1]}

\DeclareMathOperator{\argmin}{argmin}        
\newcommand{\argminunder}[1]{\underset{#1}{\argmin}}
\DeclareMathOperator{\argmax}{argmax}        

\DeclareMathOperator{\diag}{diag}       

\newcommand{\kernel}[1]{{\rm Null}\left( #1\right)}
\newcommand{\range}[1]{{\rm Range}\left( #1\right)}

\newcommand{\expSB}[2]{{ \mathbb{E}}_{#1}\left[#2\right] } 
\newcommand{\E}[1]{\mathbb{E}\left[#1\right] } 
\newcommand{\EE}[2]{\mathbb{E}_{#1}\left[#2\right] } 
 
\newcommand{\VAR}[2]{\mathbb{VAR}_{#1}\left[#2\right] }

\theoremstyle{plain}
\newtheorem{theorem}{Theorem}  
\newtheorem{lemma}[theorem]{Lemma} 
\theoremstyle{definition}
\newtheorem{assumption}{Assumption} 
\newtheorem{definition}{Definition}

\begin{document}

\maketitle
                                                                  
\begin{abstract}
We present new adaptive sampling rules for the sketch-and-project method for solving linear systems. To deduce our new sampling rules, we first show how the progress of one step of the sketch-and-project method depends directly on a \emph{sketched residual}. Based on this insight, we derive a 1) max-distance sampling rule, by sampling the sketch with the largest sketched residual 2) a proportional sampling rule, by sampling proportional to the sketched residual, and finally 3) a capped sampling rule. The capped sampling rule is a generalization of the recently introduced adaptive sampling rules for the Kaczmarz method~\cite{BaiWuSISC2018}. We provide a global linear convergence theorem for each sampling rule and show that the max-distance rule enjoys the fastest convergence. This finding is also verified in extensive numerical experiments that lead us to conclude that the max-distance sampling rule is superior both experimentally and theoretically  to the capped sampling rule. We also provide numerical insights into implementing the adaptive strategies so that the per iteration cost is of the same order as using a fixed sampling strategy when the number of sketches times the sketch size is not significantly larger than the number of columns.
\end{abstract}



\renewcommand\arraystretch{2}

 \section{Introduction}

 We consider the fundamental problem of finding an approximate solution to the linear system
\begin{equation} \label{eqn:problin}
\mA x = b,
\end{equation}
where $\mA \in \R^{m \times n}$ and $b \in \R^m.$ Given the possibility of multiple solutions, we set out to find a least-norm solution given by
\begin{equation} \label{eqn:prob}
x^* \eqdef \min_{x \in \R^n} \tfrac{1}{2}\norm{x}_{\mB}^2 \quad \mbox{subject to}\quad \mA x= b,
\end{equation}
where $\mB \in \R^{n \times n}$ is a symmetric positive definite matrix and $\norm{x}_{\mB}^2 \eqdef \dotprod{\mB x,x}.$ Here, we consider consistent systems, for which there exists an $x$ that satisfies \Cref{eqn:problin}.

When the dimensions of $\mA$ are large, direct methods for solving~\Cref{eqn:prob} can be infeasible, and  iterative methods are favored. In particular, Krylov methods including the conjugate gradient algorithms~\cite{Hestenes1952} are the industrial standard so long as one can afford full matrix vector products and the system matrix fits in memory. On the other hand, if a single matrix vector product is considerably expensive, or $\mA$ is too large to fit in memory, then randomized methods such as the randomized Kaczmarz~\cite{Kaczmarz1937,Strohmer2009} and coordinate descent method~\cite{Ma2015,leventhal2010randomized} are effective. 

\subsection{Randomized Kacmarz}\label{subsec:RK_intro}
The randomized Kaczmarz method is typically used to solve linear systems of equations in the large data regime, i.e.\ when the number of samples $m$ is much larger than the dimension $n$. The Kaczmarz method was originally proposed in 1937 and has seen applications in computer tomography (CT scans), signal processing, and other areas \cite{Kaczmarz1937,Strohmer2009,GBH70:Algebraic-Reconstruction,Nat01:Mathematics-Computerized}. 
In each iteration $k$, the current iterate $x^k$ is projected onto the solution space of
a selected row of the linear system of \Cref{eqn:problin}. Specifically, at each iteration
\begin{equation*}
    x^{k+1} = \argminunder{x\in \R^n} {\norm{x-x^k}^2} 
    \quad \mbox{subject to} \quad
    \mA_{i_k:} x = b_{i_k},
\end{equation*}
where $\mA_{i_k:}$ is the row of $\mA$ selected at iteration $k$. 
The Kaczmarz update can be written explicitly as 
\begin{equation}\label{eqn:RK_update}
    x^{k+1} = x^k + \frac{b_{i_k} - \langle \mA_{i_k:},x^k\rangle}{\norm{\mA_{i_k:}}_2^2} \mA_{i_k:}^\top.
\end{equation} 

\subsection{Coordinate descent}\label{subsec:CD_intro} 
Coordinate descent is commonly used for optimizing general convex optimization functions when the dimensions are extremely large, since at each iteration only a single coordinate (or dimension) is updated~\cite{Richtarik2014a,Richtarik2013a}. 
Here, we consider coordinate descent applied to \Cref{eqn:prob}. In this setting, it is sometimes referred to as randomized Gauss-Seidel~\cite{Ma2015,leventhal2010randomized}. 

At iteration $k$ a dimension $i\in\{1,\ldots,n\}$ is selected and the coordinate $x^k_i$ of the current iterate $x^k$ is updated such that the least-squares objective $\norm{b-\mA x}^2$ is minimized. More formally,
\begin{equation*}
    x^{k+1} =  \argminunder{x\in \R^n, \, \lambda \in \R}\norm{b-\mA x}^2 \quad \mbox{subject to} \quad x = x^k + \lambda \, e^i,
\end{equation*}
where $e^i$ is the $i\thup$ coordinate vector.
Let $\mA_{:i}$ denote the $i\thup$ column of $\mA$. The explicit update for coordinate descent applied to \Cref{eqn:prob} is given by 
\begin{equation}\label{eqn:cd_update}
    x^{k+1} = x^k - \frac{\mA_{:i_k}^\top (\mA x^k-b)}{\norm{\mA_{:i_k}}}e^{i_k}.
\end{equation}

\subsection{Sketch-and-project methods}\label{subsec:intro_snp}
Sketch-and-project is a general archetypal algorithm that unifies a variety of randomized iterative methods including both randomized Kaczmarz and coordinate descent along with all of their block variants~\cite{Gower2015}. At each iteration, sketch-and-project methods project the current iterate onto a subsampled or sketched linear system with respect to some norm. Let $\mB \in \R^{n\times n}$ be a positive definite matrix. We will consider the projection with respect to the $\mB$--norm given by  $\norm{\cdot}_{\mB} = \ve{\cdot}{\mB \cdot}$.

 Let $\mS_i \in \R^{m \times \tau}$ for $i=1,\ldots, q$ be the set of \emph{sketching matrices} where $\tau \in \N$ is the \emph{sketch size}.
 In general, the set of sketching matrices $\mS_i$ could be infinite, however, here, we restrict ourselves to a finite set of $q \in \N$ sketching matrices. At the $k\thup$ iteration of the sketch-and-project algorithm,  a sketching matrix $\mS_{i}$ is selected and the current iterate $x^{k}$ is projected onto the solution space of the sketched system $\mS_{i_k}^\top \mA x = \mS_{i_k}^\top b$ with respect to the $\mB$--norm. Given a selected index $i_k \in \{1,\ldots, q\}$ the sketch-and-project update solves
\begin{equation}
\quad x^{k+1}  =  \argminunder{x\in \R^n} \norm{x- x^{k}}_\mB^2 \quad \mbox{subject to} \quad  \mS_{i_k}^\top \mA x = \mS_{i_k}^\top b. \quad \label{eqn:NF}
\end{equation}
The closed form solution to~\Cref{eqn:NF} is given by
\begin{equation} \label{eqn:xupdate}
x^{k+1} = x^k - \mB^{-1}\mA^\top \mH_{i_k}(\mA x^k-b),
\end{equation}
where
\begin{equation} \label{eq:Hi}
\mH_i \eqdef \mS_i (\mS_{i}^\top \mA \mB^{-1}\mA^\top \mS_{i})^{\dagger} \mS_{i}^\top, \quad \mbox{for }i =1,\ldots, q,
\end{equation}
and $\dagger$ denotes the pseudoinverse.

One can recover the randomized Kaczmarz method under the sketch-and-project framework by choosing the matrix $\mB$ as the identity matrix and sketches $\mS_i = e^i$. If instead $\mB = \mA^\top \mA$ and sketches $\mS_i = \mA e^i = \mA_{:i}$, where $\mA_{:i}$ is the $i\thup$ column of the matrix $\mA$, then the resulting method is coordinate descent.

\subsection{Sampling of indices}\label{subsec:intro_adasample}
An important component of the methods above is the selection of the index $i_k$ at iteration $k$. Methods often use independently and identically distributed (i.i.d.) indices, as this choice makes the method and analysis relatively simple \cite{Strohmer2009,nesterov2012efficiency}. In addition to choosing indices i.i.d.\ at each iteration, several adaptive sampling methods have also been proposed, which we discuss next. These sampling strategies use information about the current iterate in order to improve convergence guarantees over i.i.d.\ random sampling strategies at the cost of extra calculation per iteration. Under certain conditions, such strategies can be implemented with only a marginal additional cost per iteration.

\subsubsection{Sampling for the Kaczmarz method}
The original Kaczmarz method cycles through the rows of the matrix $\mA$ and makes projections onto the solution space with respect to each row~\cite{Kaczmarz1937}.  In 2009, Strohmer and Vershynin suggested selecting rows with probabilities that are proportional to the squared row norms (i.e.\ $p_i \propto \norm{\mA_{i:}}_2^2$) and provided the first proof of exponential convergence of the randomized Kaczmarz method \cite{Strohmer2009}. 

Several adaptive selection strategies have also been proposed in the Kaczmarz setting. The max-distance Kaczmarz or Motzkin's method selects the index $i_k$ at iteration $k$ that leads to the largest magnitude update~\cite{Nutini2016,motzkin1954relaxation}. In addition to the max-distance selection rule, Nutini et al also consider the greedy selection rule that chooses the row corresponding to the maximal residual component i.e.\ $i_k = \argmax_i |\mA_{i:} x^k - b_i|$ at each iteration, but show that the max-distance Kacmzarz method performs at least as well as this strategy \cite{Nutini2016}. More complicated adaptive methods have also been suggested for randomized Kaczmarz, such as the capped sampling strategies proposed in \cite{BaiWuSISC2018,BAIWu201821} or the Sampling Kaczmarz Motzkin's method of \cite{DeLoeraMotzkin16}. 

\subsubsection{Sampling for coordinate descent}
For coordinate descent, several works have investigated adaptive coordinate selection strategies~\cite{perekrestenko2017fasterCD,nutini2015coordinate,nesterov2012efficiency,khalil-abid-gower2018}. As coordinate descent is not restricted to solving linear systems, these works often consider more general convex loss functions. 
A common greedy selection strategy for coordinate descent applied to differentiable loss functions is to select the coordinate that corresponds to the maximal gradient component, which is known as the Gauss-Southwell rule \cite{tseng1990dual,luo1992convergence,nutini2015coordinate, nesterov2012efficiency} or adaptively according to a duality gap~\cite{Csiba2015}.

\subsubsection{Sampling for sketch-and-project}
The problem of determining the optimal fixed probabilities with which to select the index $i_k$ at each iteration $k$ was shown in Section 5.1 of~\cite{Gower2015} to be a convex semi-definite program, which is often a harder problem than solving the original linear system. The problem of determining the optimal adaptive probabilities is even harder as one must consider the effects of the current index selection on the future iterates. Here, instead, we present adaptive sampling rules that are not necessarily optimal, but can be efficiently implemented and are proven to converge faster than the fixed non-adaptive rules.

\section{Contributions}\label{subsec:contr}
Adaptive sampling strategies have not yet been analyzed for the general sketch-and-project framework.
We introduce three different adaptive sampling rules for the general sketch-and-project method: max-distance, the capped-adaptive sampling rule, and proportional sampling probabilities. 
We prove that each of these methods converge exponentially in mean squared error with convergence guarantees that are strictly faster than the guarantees for sampling indices uniformly. 

\subsection{Key quantity: Sketched loss}

As we will see in the general convergence analysis of the sketch-and-project method detailed in \Cref{sec:convergence}, the convergence at each iteration depends on the current iterate $x^k$ and a key quantity known as the sketched loss 
\begin{equation}\label{eqn:residualk}
 f_i(x^k) \eqdef  \norm{ \mA x^k-b}_{\mH_i}^2,
 \end{equation}
of the sketch $\mS_i$ (recall the definition of $\mH_i$ in \Cref{eq:Hi}). This sketched loss was introduced in~\cite{Richtarik2017stochastic} where the authors show that the sketch-and-project method can be seen as a stochastic gradient method (we expand on this in)~\Cref{sec:reform_SGD}.  We show that using adaptive selection rules based on the sketched losses results in new methods with a faster convergence guarantees. 

\subsection{Max-distance rule}
We introduce the max-distance sketch-and-project method, which is a generalization of both the max-distance Kaczmarz method (also known as Motzkin's method)~\cite{Nutini2016,motzkin1954relaxation,HadNeeMotz18}, greedy coordinate descent (Gauss-Southwell rule~\cite{nutini2015coordinate}), and all their possible block variants. Nutini et al.\ showed that the max-distance Kaczmarz method performs at least as well as uniform sampling and the non-uniform sampling method of~\cite{Strohmer2009}, in which rows are sampled with probabilities proportional to the squared row norms of $\mA$~\cite{Nutini2016}. We extend this result to the general sketch-and-project setting and also show that the max-distance rule leads to a convergence guarantee that is \emph{strictly} faster than that of any fixed probability distribution.

\subsection{The capped adaptive rule}
A new family of adaptive sampling methods were recently proposed for the Kaczmarz type methods ~\cite{BaiWuSISC2018,BAIWu201821}. We extend these methods to the sketch-and-project setting, which allows for their application in other settings such as for coordinate descent. While introduced under the names greedy randomized Kaczmarz and relaxed greedy randomized Kaczmarz, we refer to these methods as \emph{capped adaptive} methods because they 
select indices $i$ whose corresponding sketched losses $f_i(x^k)$ are larger than a capped threshold given by a convex combination of the largest and average sketched losses. It was proven in~\cite{BaiWuSISC2018} that the convergence guarantee when using the capped adaptive rule is strictly faster than the fixed non-uniform sampling rule given in~\cite{Strohmer2009}. In \Cref{subsec:BaiWuExt}, we generalize this capped adaptive sampling to sketch-and-project methods and prove that the resulting convergence guarantee of this adaptive rule is slower than that of the max-distance rule. Furthermore, in \Cref{sec:sample_spec_costs}, we show that the max-distance rule requires less computation at each iteration than the capped adaptive rule.

\subsection{The proportional adaptive rule}
We also present a new and much simpler randomized adaptive rule as compared to the capped adaptive rule discussed above, in which indices are sampled with probabilities that are directly \emph{proportional} to their corresponding sketched losses $f_i(x^k)$. We show that this rule gives a resulting convergence that is at least twice as fast as when sampling the sketches uniformly.

\subsection{Efficient implementations}\label{subsec:eff_imp}

Our adaptive methods come with the added cost of computing the sketched loss $f(x^k)$ of \Cref{eqn:residualk} at each iteration. Fortunately, the sketched loss can be computed efficiently with certain precomputations as discussed in \Cref{sec:imp_tricks}. We show how the sketched losses can be maintained efficiently via an auxiliary update, leading to reasonably efficient implementations of the adaptive sampling rules. 
We demonstrate improved performance of the adaptive methods over uniform sampling when solving linear systems with both real and synthetic matrices per iteration and in terms of the flops required. 

\subsection{Consequences and future work}
Our results on adaptive sampling have consequences on many other closely related problems. For instance, 
an analogous sampling strategy to our proportional adaptive rule has been proposed for coordinate descent in the primal-dual setting for optimizing regularized loss functions~\cite{perekrestenko2017fasterCD}. Also a  variant of adaptive and greedy coordinate descent has been shown to speed-up the solution of the matrix scaling problem \cite{khalil-abid-gower2018}. The matrix scaling problem is equivalent to an entropy-regularized version of the optimal transport problem which has numerous applications in machine learning and computer vision~\cite{khalil-abid-gower2018,Cuturi2013_4927}. 
Thus the adaptive methods proposed here may be extended to these other settings such as adaptive coordinate descent for more general smooth optimization~\cite{perekrestenko2017fasterCD}.
The adaptive methods and  the analysis  proposed in this paper may also provide insights toward adaptive sampling for other classes of optimization methods such as stochastic gradient, since the randomized Kaczmarz method can be reformulated as stochastic gradient descent applied to the least-squares problem~\cite{NeedellWard2015}.

\section{Notation}

We now introduce notation that will be used throughout.  Let $\Delta_q$ denote the simplex in $\R^q$, that is
\begin{equation*}
\Delta_q \eqdef \{ p \in \R^q \; : \; \sum_{i=1}^q p_i =1, \; p_i \geq 0, \; \mbox{for }i=1,\ldots, q \}.
\end{equation*}
For probabilities $p \in \Delta_q $ and values $x_i$ depending on an index $i = 1,\ldots, q$, we denote 
$\EE{i \sim p}{x_i} \eqdef \sum_{i=1}^q p_i x_i,$
where $i\sim p$ indicates that $i$ is sampled with probability $p_i$. 
At the $k\thup$ iteration of the sketch-and-project algorithm,  a sketching matrix $\mS_{i_k}$ is sampled with probability
\begin{equation}\label{eq:pik}
     \Prob{\mS_{i_k} = \mS_i\; | \; x^k} =  p_i^k, \quad \mbox{for }i=1,\ldots, q,
\end{equation}
where $p^k \in \Delta_q$ and we use $p^k \eqdef (p_1^k,\ldots, p_q^k)$ 
  to denote the vector containing these probabilities. We drop the superscript  $k$ when the probabilities do not depend on the iteration.

For any positive semi-definite matrix $\mG$ we write the norm induced by $\mG$ as 
$
\norm{\cdot}_{\mG}^2 \eqdef \ve{\cdot}{\mG \cdot},
$
while $\norm{\cdot}$ denotes the standard 2-norm ($\norm{\cdot}_2$). For any matrix $\mM$, $\norm{\mM}_F \eqdef \sqrt{\sum_{i,j} \mM_{ij}^2}$. We use 
\begin{equation*}
\lambda_{\min}^+(\mG) \eqdef  \min_{v\in \range{\mG}} \frac{\norm{v}_{\mG}^2}{\norm{v}_2^2},
\end{equation*}
to denote the smallest non-zero eigenvalue of $\mG.$

\subsection{Organization}
The remainder of the paper is organized as follows. \Cref{sec:reform_SGD,sec:geomMotiv} provide additional background on the sketch-and-project method and motivation for adaptive sampling in this setting. \Cref{sec:reform_SGD} explains how the sketch-and-project method can be reformulated as stochastic gradient descent. The sampling of the sketches can then be seen as importance sampling in the context of stochastic gradient descent. \Cref{sec:geomMotiv} provides geometric intuition for the sketch-and-project method and motivates why one would expect adaptive sampling strategies that depend on the sketched losses $f_i(x^k)$ to perform well. 

\Cref{sec:methods} introduces the various sketch selection strategies considered throughout the paper, while \Cref{sec:convergence} provides convergence guarantees for each of the resulting methods. In \Cref{sec:imp_tricks}, we discuss the computational costs of adaptive sketch-and-project for the sketch selection strategies of \Cref{sec:methods} and suggest efficient implementations of the methods. \Cref{sec:costs-and-convergence-summaries} discusses convergence and computational cost for the special subcases of randomized Kaczmarz and coordinate descent. Performance of adaptive sketch-and-project methods are demonstrated in \Cref{sec:experiments} for both synthetic and real matrices.

\section{Reformulation as importance sampling for stochastic gradient descent}\label{sec:reform_SGD}
The sketch-and-project method can be reformulated as a stochastic gradient method, as shown in~\cite{Richtarik2017stochastic}. We use this reformulation to motivate our adaptive sampling as a variant of importance sampling.

Let $p \in \Delta_q$. Consider the stochastic program
\begin{equation}
    \min_{x \in \R^d} F(x) \; \eqdef \; \EE{i \sim p}{f_i(x)} \;= \;\EE{i \sim p}{\norm{\mA x - b}_{\mH_i}}^2.
    \label{eqn:stoch_obj}
\end{equation} 
Objective functions $F(x)$ such as the one in~\Cref{eqn:stoch_obj} are common in machine learning, where $f_i(x)$ often represents the loss with respect to a single data point.

When  $\EE{i \sim p}{\mH_i}$ is invertible, solving~\Cref{eqn:stoch_obj} is equivalent to solving the linear system~\Cref{eqn:problin}. 
This invertibility condition on $\EE{i \sim p}{\mH_i}$ can be significantly relaxed by using the following technical exactness assumption on the probability $p$ and the set of sketches introduced in~\cite{Richtarik2017stochastic}. 
\begin{assumption}\label{ass:exact}
Let $p\in \Delta_q$, $\Sigma \eqdef \{S_1,\ldots, S_q\}$ be a set of sketching matrices and $\mH_i$ as defined in \Cref{eq:Hi}. 
We say that the exactness assumption holds for $(p,\Sigma )$ if
\begin{equation*}
    \kernel{\EE{i \sim p}{\mH_i}} \subset \kernel{\mA}.
\end{equation*}
\end{assumption}

This exactness assumption guarantees\footnote{This can be shown by applying~\Cref{lem:NullA}  in \Cref{sec:aux_lemmas} with with $\mG = \EE{i \sim p}{\mH_i}$ and $\mW = \mA$.} that
\begin{equation} \label{eq:98js84js8j4}
\kernel{\mA} = \kernel{\mA^\top \EE{i \sim p}{\mH_i} \mA}.
\end{equation}
This in turn guarantees that the expected sketched loss of the point $x$ is zero if and only if $\mA x = b$. 
Indeed, by taking the derivative of~\eqref{eqn:stoch_obj} and setting it to zero we have that
\begin{eqnarray*}
\nabla F(x) \;=\; \mA^\top \EE{i \sim p}{\mH_i}(\mA x-b) \;=\;
\mA^\top \EE{i \sim p}{\mH_i}\mA( x-x^*) \;=\; 0.
\end{eqnarray*}
Thus, every minimizer $x$ of Equation~\eqref{eqn:stoch_obj} is such that 
\begin{eqnarray}
 x-x^* &\in & 
 \kernel{\mA^\top \EE{i \sim p}{\mH_i}\mA}
\; \overset{\eqref{eq:98js84js8j4} }{=}\; \kernel{\mA},\label{eqn:kernelAEAA}
\end{eqnarray}
thus $\mA(x-x^*) = \mA x -b = 0$.
  As shown in~\cite{Gower2015c} and~\cite{Richtarik2017stochastic} this exactness assumption holds trivially for most practical sketching techniques.

When the number of $f_i$ functions is large, the SGD (stochastic gradient descent) method is typically the method of choice for solving \Cref{eqn:stoch_obj}. 
To view the sketch-and-project update in \Cref{eqn:xupdate} as a SGD method, we sample an index $i_k \sim p$ at each iteration and takes a step
\begin{equation}
x^{k+1} = x^k -  \nabla^{\mB} f_{i_k}(x^k),
    \label{eqn:sgd_update}
\end{equation}
where $ \nabla^{\mB} f_{i_k}(x^k)$ is the gradient taken with respect to the $\mB$--norm.
 For $f_i(x^k)$ of \Cref{eqn:residualk}, the exact expression of this stochastic gradient is given by
\begin{equation} \label{eq:aj98j8aj3}
\nabla^{\mB} f_{i_k}(x^k) =  \mB^{-1}\mA^\top \mH_{i_k}(\mA x^k-b).
\end{equation}
By plugging~\Cref{eq:aj98j8aj3} into~\Cref{eqn:sgd_update} we can see that the resulting update is equivalent to a the sketch-and-project update in \Cref{eqn:xupdate}.

Though the indices $i\in [1,\dots, q]$ are often sampled uniformly at random for SGD,  many alternative sampling distributions have been proposed in order to accelerate convergence, including adaptive sampling strategies  \cite{csiba2018importance, johnson2013accelerating,NeedellWard2015,zhao2015stochastic,katharopoulos2018not,loshchilov2015online,alain2015variance}. Such sampling strategies give more weight to sampling indices corresponding to a larger loss $f_{i}(x)$ or a larger gradient norm $ \norm{\nabla^{\mB} f_{i}(x)}^2.$
In the sketch-and-project setting, it is not hard to show\footnote{See Lemma 3.1 in~\cite{Richtarik2017stochastic}.} that these two sampling strategies result in similar methods since 
\[f_{i}(x) = \norm{\mA x-b}_{\mH_i}^2 = \tfrac{1}{2}\norm{\nabla^{\mB} f_{i}(x)}_{\mB}^2.\]

In general, updating the loss and gradient of every $f_{i}(x)$ at each iteration can be too expensive. Thus many methods resort to using global approximations of these values such as the Lipschitz constant of the gradient~\cite{NeedellWard2015} that lead to fixed data-dependent sample distributions. For the sketch-and-project setting, 
 we demonstrate in~\Cref{sec:imp_tricks} that the adaptive sample distributions can be calculated efficiently, with a per-iterate cost on the same order as is required for the sketch-and-project update.

\section{Geometric viewpoint and motivational analysis} 
\label{sec:geomMotiv}
\def\dotMarkRightAngle[size=#1](#2,#3,#4){%
 \draw ($(#3)!#1!(#2)$) -- 
       ($($(#3)!#1!(#2)$)!#1!90:(#2)$) --
       ($(#3)!#1!(#4)$);
 \path (#3) --node[circle,fill,inner sep=.5pt]{} ($($(#3)!#1!(#2)$)!#1!90:(#2)$);
}
\tdplotsetmaincoords{70}{200}
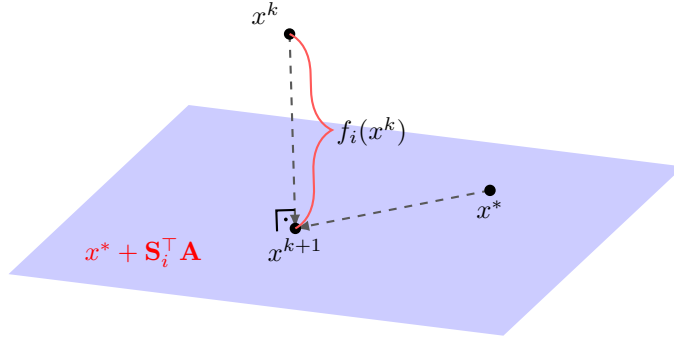
\begin{figure}
    \centering
    \begin{tikzpicture}[scale = 3.5,thick,tdplot_main_coords]
    \coordinate (A1) at (0,0,0);
    \coordinate (A2) at (2,0,0);
    \coordinate (A3) at (2,2,0);
    \coordinate (A4) at (0,2,0);
    
    \coordinate (S) at (0.6,0.5,0);
    \coordinate (E) at (0.5,3,1.5);
    \coordinate (P) at (1.15,1.15,0);
    \fill[fill=blue!20] (A1) -- (A2) -- (A3)  -- (A4) -- cycle;
    \draw[black!65,-latex, dashed] (E) -- (P);
    \draw[black!65,-latex, dashed] (S) -- (P);
    \draw[black,fill, thick]
    (E) circle (0.5pt) node[above left] {$x^k$};
    \draw[black,fill, thick]
    (P) circle (0.5pt) node[below] {$x^{k+1}$};
    \draw[black,fill, thick]
    (S) circle (0.5pt) node[below] {$x^{*}$};
    \node[red] at ($(A3)+(-0.4,-0.4,0)$) {$x^* + \null{\mS^\top_{i}\mA}$};
    \dotMarkRightAngle[size=2pt](E,P,A3);
     \draw [decorate,red!65, decoration={brace,amplitude=15pt},
     xshift=20pt,yshift=0pt]
    (E) -- (P) node [black,midway,xshift=30pt] 
    {$f_i(x^k)$};
     \end{tikzpicture}
     \caption{The geometric interpretation of~\Cref{eqn:NF}, as the projection of $x^k$ onto a random affine space that contains $x^*.$ The distance traveled is given by $f_{i}(x^k) = \norm{x^{k+1}-x^k}_{\mB}^2.$ }
    \label{fig:SK}
\end{figure}

The sketch-and\linebreak-project method given in \Cref{eqn:NF} can be seen as a method that calculates the next iterate $x^{k+1}$ by projecting the previous iterate $x^k$ onto a random affine space. 
Indeed, the constraint in~\Cref{eqn:NF} can be re-written as 
\begin{equation}\label{eq:affineSA}
     \{x \;: \; \mS_{i}^\top \mA x = \mS_{i}^\top b \} \quad = \quad x^* + \kernel{\mS_{i}^\top \mA}.
\end{equation}

In particular, \Cref{eqn:NF} is an orthogonal projection of the point $x^k$ onto an affine space that contains $x^*$ with respect to the $\mB$--norm. See~\Cref{fig:SK} for an illustration.
This projection is determined by the following projection operator.

\begin{lemma} \label{lem:Z}
Let
\begin{equation}\label{eqn:Zk}
\mZ_{i}  \eqdef 
 \mB^{-1/2}\mA^\top  \mS_i (\mS_{i}^\top \mA \mB^{-1}\mA^\top \mS_{i})^{\dagger} \mS_{i}^\top \mA \mB^{-1/2} =  \mB^{-1/2}\mA^\top \mH_{i} \mA \mB^{-1/2}, \quad  
\end{equation}
for $i = 1,\ldots, q,$ which is the orthogonal projection matrix onto $\range{\mB^{-1/2} \mA^\top \mS_i}.$ Consequently
\begin{equation}
\mZ_{i} \mZ_{i} = \mZ_{i}, \quad \mbox{and} \quad (\mI-\mZ_{i})\mZ_{i} =0.\label{eq:Ziproj}
\end{equation}
Furthermore we have that $(\mI-\mZ_i)$ gives the projection depicted in Figure~\ref{fig:SK} since
\begin{equation} \label{eqn:convstep1}
\mB^{1/2}(x^{k+1} - x^*)  \; = \; (\mI- \mZ_{i_k}) \mB^{1/2} (x^k - x^*).
\end{equation}
Finally we can re-write the sketched loss as  
\begin{equation}\label{eqn:fiBnorm}
f_{i}(x) \;=\; \|\mB^{1/2}(x-x^*)\|_{\mZ_i}^2, \quad \mbox{for }i=1,\ldots, q.
\end{equation}
\end{lemma}
\begin{proof}
The proof of~\Cref{eq:Ziproj} relies on standard properties of the pseudoinverse and is given in Lemma 2.2 in~\cite{Gower2015}.

As for the proof of~\Cref{eqn:convstep1}, subtracting $x^*$ from both sides of \Cref{eqn:xupdate} we have that 
\begin{align} 
x^{k+1} - x^*  &\hspace*{1em} = \hspace*{1em}x^k - x^* - \mB^{-1}\mA^\top \mH_{i_k}(\mA x^k-b)\nonumber \\
& \hspace*{1em} \overset{\mathclap{\mA x^* = b}}{=} \hspace*{1em} x^k - x^* - \mB^{-1/2}\mB^{-1/2}\mA^\top \mH_{i_k} \mA \mB^{-1/2} \mB^{1/2}(x^k-x^*) \nonumber\\
&\hspace*{1em} \overset{\mathclap{\eqref{eqn:Zk}}}{=} \hspace*{1em} x^k - x^* - \mB^{-1/2}\mZ_{i_k} \mB^{1/2} (x^k - x^*). \label{eqn:convstep12}
\end{align}
It now only remains to multiply both sides by $\mB^{1/2}.$

Finally the proof of~\Cref{eqn:fiBnorm} follows by using $\mA x^* =b$ together with the definitions of $\mH_i$ and $\mZ_i$ given in~\Cref{eq:Hi} and~\Cref{eqn:Zk} so that
 \begin{equation}
    f_i(x) = \norm{\mA(x-x^*)}_{\mH_i}^2   = \norm{x-x^*}_{\mA^\top \mH_i\mA}^2 
\overset{\eqref{eqn:Zk}}{=}\norm{\mB^{1/2}(x-x^*)}_{\mZ_i}^2.
\end{equation}
\end{proof}
With the explicit expression for the projection operator we can calculate the progress made by a single iteration of the sketch-and-progress method. The convergence proofs later on in~\Cref{sec:convergence} will rely heavily on \Cref{lem:geometry,lem:motivation}.

\begin{lemma}\label{lem:geometry}
Let $x^k \in \R^d$ and  let $x^{k+1}$ be given by~\Cref{eqn:NF}. Then the squared magnitude of the update is 
\begin{equation}
   \norm{x^{k+1} -x^k}_{\mB}^2 = f_{i_k}(x^k), \label{eq:distravel} 
\end{equation}
and the error from one iteration to the next decreases according to 
\begin{equation}
    \norm{x^{k+1} -x^*}_{\mB}^2 = \norm{x^{k} -x^*}_{\mB}^2 - f_{i_k}(x^k). \label{eq:onestepprog}
\end{equation}
\end{lemma}
\begin{proof}
We begin by deriving \Cref{eq:onestepprog}. 
Taking the squared norm in \Cref{eqn:convstep1} we have
\begin{align}
\norm{x^{k+1} - x^*}_{\mB}^2 
\hspace*{.5em} &=\hspace*{.5em}\norm{(\mI - \mB^{-1/2}\mZ_{i_k} \mB^{1/2})  (x^k - x^*)}_{\mB}^2 \nonumber \\
&=\hspace*{.5em}\norm{(\mI - \mZ_{i_k} )  \mB^{1/2}(x^k - x^*)}_2^2 \nonumber \\
& = \hspace*{.5em}\dotprod{\mB^{1/2}(x^k - x^*),(I-\mZ_{i_k})(I-\mZ_{i_k})\mB^{1/2}(x^k - x^*)} \nonumber \\
& \overset{\mathclap{\eqref{eq:Ziproj}}}{=} \hspace*{.5em} \dotprod{\mB^{1/2}(x^k - x^*),(I-\mZ_{i_k})\mB^{1/2}(x^k - x^*)} \nonumber \\
&=\hspace*{.5em} \norm{x^k-x^*}_{\mB}^2 - \dotprod{\mZ_{i_k} \mB^{1/2}(x^k - x^*), \mB^{1/2}(x^k - x^*)} \nonumber \\
&\overset{\mathclap{\eqref{eqn:fiBnorm}}}{=} \hspace*{.5em} \norm{x^k-x^*}_{\mB}^2 - f_i(x^k). \label{eqn:convstepmain}
\end{align}

Finally we establish~\Cref{eq:distravel} by subtracting $x^k$ from both sides of \Cref{eqn:xupdate} so that
\[ x^{k+1} - x^k  = - \mB^{-1/2}\mZ_{i_k} \mB^{1/2} (x^k - x^*).\]
It now remains to take the squared $\mB$--norm and use~\Cref{eqn:fiBnorm}.
\end{proof}

\Cref{eq:distravel} shows that the distance traveled from $x^k$ to $x^{k+1}$ is given by the sketch residual $f_{i_k}(x^k),$ as we have depicted in \Cref{fig:SK}. Furthermore, \Cref{eq:onestepprog} shows that the contraction of the error $x^{k+1}-x^*$ is given by $-f_{i_k}(x^k)$.
 Consequently 
\Cref{lem:geometry} indicates that in order to make the most progress in one step, or maximize the distance traveled,  we should choose $i_k$ corresponding to the largest sketched loss $f_{i_k}(x^k)$. We refer to this greedy sketch selection as 
the max-distance rule, which we explore in detail in ~\Cref{sec:maxdist}.


Next we give the expected decrease in the error.
\begin{lemma} \label{lem:motivation}
Let $p^k \in \Delta_q$.
Consider the iterates of the sketch-and-project method given in~\Cref{eqn:xupdate} where  $i_k \sim p_i^k$ as is done in~\Cref{alg:adaSKep}.
It follows that 
\begin{eqnarray*}
\EE{i \sim p^k}{\norm{x^{k+1}-x^*}_{\mB}^2 \; | \; x^k } &=& \norm{x^{k}-x^*}_{\mB}^2  - \EE{i \sim p^k}{f_i(x^k)}.
\end{eqnarray*} 
\end{lemma}
\begin{proof}
The result follows by taking the expectation over~\Cref{eq:onestepprog} conditioned on $x^k$.
\end{proof}

\Cref{lem:motivation} suggests choosing adaptive probabilities so that $\EE{i \sim p^k}{f_i(x^k)}$ is large. This analysis motivates the adaptive methods described in \Cref{sec:adaptive}.

\section{Selection rules}\label{sec:methods}

Motivated by \Cref{lem:geometry,lem:motivation}, we might think that sampling rules that prioritize larger entries of the sketched loss should converge faster. 
From this point we take two alternatives, 1) choose the $i_k$ that maximizes the decrease (\Cref{sec:maxdist}) or 2) choose a probability distribution that prioritizes the biggest decrease (\Cref{sec:adaptive}). Below, we describe several sketch-and-project sampling strategies (fixed, adaptive, and greedy) and analyze their convergence in \Cref{sec:convergence}. 
The adaptive and greedy sampling strategies require knowledge of the current sketched loss vector at each iteration. Calculating the sketched loss from scratch is expensive, thus in~\Cref{sec:imp_tricks} we will show how to efficiently calculate the new sketched loss $f(x^{k+1})$ using the previous sketched loss $f(x^{k})$.

\subsection{Fixed sampling}\label{subsec:fixedProbMethod}
We first recall the standard non-adaptive sketch-and-project method that will be used as a comparison for the greedy and adaptive versions. In the non-adaptive setting the sketching matrices are sampled from a fixed distribution that is independent of the current iterate $x^k$. For reference, the details of the non-adaptive sketch-and-project method are provided in \Cref{alg:fixedProbs}.
 
\begin{algorithm}
\begin{algorithmic}[1]
\State \textbf{input:}  $x^0\in \R^{n},$ $\mA\in \R^{m\times n}$, $b\in\R^m$, $p\in \Delta_q$, and a set of sketching matrices $\mS = [\mS_1,\dots, \mS_q]$
\For {$k = 0, 1, 2, \dots$}
    \State $i_k \sim p_i$ 
	\State $x^{k+1} = x^k - \mB^{-1}\mA^\top \mH_{i_k}(\mA x^k-b)$
\EndFor
\State \textbf{output:} last iterate $x^{k+1}$
\end{algorithmic}
\caption{Non-Adaptive Sketch-and-Project}
\label{alg:fixedProbs}
\end{algorithm}

\subsection{Adaptive probabilities}\label{sec:adaptive}

\Cref{eq:onestepprog} motivates selecting indices that correspond to larger sketched losses with higher probability. We refer to such sampling strategies as adaptive sampling strategies, as they depend on the current iterate and its corresponding sketched loss values. In the adaptive setting, we sample indices at the $k\thup$ iteration with probabilities given by $p^k \in \Delta_q$. Adaptive sketch-and-project is detailed in \Cref{alg:adaSKep}.  

\begin{algorithm}
\begin{algorithmic}[1]
\State \textbf{input:}  $x^0\in \R^{n},$ $\mA\in \R^{m\times n}$, $b\in\R^m$, and a set of sketching matrices $\mS = [\mS_1,\dots, \mS_q]$
\For {$k = 0, 1, 2, \dots$}
	\State  $f_i(x^k) = \norm{\mA x^k-b}_{\mH_i}$ for $i=1,\ldots, q$ 
    \State  Calculate $p^k \in \Delta_q$  \Comment Typically based on $f(x^k)$
    \State  $i_k \sim p_i^k$ 
	\State  $x^{k+1} = x^k - \mB^{-1}\mA^\top \mH_{i_k}(\mA x^k-b)$ \label{ln:updaters}
\EndFor
\State \textbf{output:} last iterate $x^{k+1}$
\end{algorithmic}
\caption{Adaptive Sketch-and-Project}
\label{alg:adaSKep}
\end{algorithm}

\subsection{Max-distance rule}\label{sec:maxdist}
We refer to the greedy sketch selection rule given by
\begin{equation}\label{eqn:max_dist}
i_k = \argmax_{i=1,\ldots, q} f_i(x^k) =  \norm{\mA x^k-b}_{\mH_i}^2,
\end{equation}
as the max-distance selection rule. Per iteration, the max-distance rule leads to the best expected decrease in mean squared error. The max-distance sketch-and-project method is described in \Cref{alg:maxDist}.
 This greedy selection strategy has been studied for several specific choices of $\mB$ and sketching methods. For example, in the Kaczmarz setting, this strategy is typically referred to as max-distance Kaczmarz or Motzkin's method~\cite{griebel2012greedy,Nutini2016,motzkin1954relaxation}. For coordinate descent, this selection strategy is the Gauss-Southwell rule \cite{nesterov2012efficiency,nutini2015coordinate}. We provide a convergence analysis for the general sketch-and-project max-distance selection rule in \Cref{thm:maxDistConv}. We further show that max-distance selection leads to a convergence rate that is strictly larger than the resulting convergence rate when sampling from any fixed distribution in \Cref{thm:maxDistBetterThanUnif}.

\begin{algorithm}
\begin{algorithmic}[1]
\State \textbf{input:}  $x^0\in \R^{n},$ $\mA\in \R^{m\times n}$, $b\in\R^m$, and a set of sketching matrices $\mS = [\mS_1,\dots, \mS_q]$
\For {$k = 0, 1, 2, \dots$}
	\State $f_i(x^k) = \norm{\mA x^k-b}_{\mH_i}$ for $i=1,\ldots, q$
    \State $i_k = \arg \max_{i=1,\ldots, q} f_i(x^k)$  
	\State $x^{k+1} = x^k - \mB^{-1}\mA^\top \mH_{i_k}(\mA x^k-b)$
\EndFor
\State \textbf{output:} last iterate $x^{k+1}$
\end{algorithmic}
\caption{Max-Distance Sketch-and-Project}
\label{alg:maxDist}
\end{algorithm}

\section{Convergence}\label{sec:convergence}
We now present convergence results for the max-distance selection rule, uniform sampling, and adaptive sampling with probabilities proportional to the sketched loss. We summarize the rates of convergence discussed throughout \Cref{sec:convergence} in \Cref{tab:conv_summary}. Our first step in the analysis is to establish an invariance property of the iterates in the following lemma.\footnote{This lemma was first presented in~\cite{Gower2015c}. We present and prove it here for completeness.} In particular, \Cref{lem:invariant} guarantees the error vectors $x^k-x^*$ remain in the subspace $\range{\mB^{-1} \mA^\top}$ for all iterations if $x^0 \in \range{\mB^{-1} \mA^\top}$, which allows for a tighter convergence analysis. 
\begin{lemma}\label{lem:invariant} 
If $x^0 \in \range{\mB^{-1} \mA^\top}$ then $x^k -x^* \in \range{\mB^{-1}\mA^\top}.$
\end{lemma}
\begin{proof}
First note that $x^* \in \range{\mB^{-1}\mA^{\top}}$. This follows by taking the Lagrangian of \Cref{eqn:prob} given by
\[L(x,\lambda) = \tfrac{1}{2}\norm{x}_{\mB}^2+ \dotprod{\lambda, \mA x-b}.\]
Taking the derivative with respect to $x$, setting to zero and isolating $x$ gives
\begin{equation}\label{eqn:xstar-range}
  x^* = -\mB^{-1}\mA^\top \lambda \in \range{\mB^{-1}\mA^{\top}}.
\end{equation}
Consequently $x^* - x^0 \in  \range{\mB^{-1}\mA^{\top}}.$ Assuming that  $x^k -x^* \in \range{\mB^{-1}\mA^{\top}}$ holds, by induction we have that
\begin{equation}
x^{k+1}-x^* \overset{\eqref{eqn:xupdate}}{=}  x^k -x^*- \underbrace{\mB^{-1}\mA^\top \mS_{i_k}(\mS_{i_k}^\top \mA \mB^{-1}\mA^\top \mS_{i_k})^{\dagger}\mS_{i_k}^\top(\mA x^k-b)}_{\in \range{\mB^{-1}\mA^{\top}}}.
\end{equation}
Thus $x^{k+1}-x^*$ is the difference of two elements in the subspace $ \range{\mB^{-1}\mA^{\top}}$ and thus $x^{k+1}-x^*\in \range{\mB^{-1}\mA^{\top}}.$ 
\end{proof}

We also make use of the following fact. For a positive definite random matrix $\mM \in \R^{n\times n}$ drawn from some probability distribution $\cD$ and for any vector $v \in \R^{n}$
\begin{equation}\label{eqn:linNorm}
    \EE{\cD}{\norm{v}_{\mM}^2} 
    = \EE{\cD}{\dotprod{v,\mM v}}
    = \dotprod{v, \EE{\cD}{\mM v}}
    = \norm{v}_{\EE{\cD}{\mM}}^2.
\end{equation}

\subsection{Important spectral constants}
We define two key spectral constants in the following definition that will be used to express our forthcoming rates of convergence.
\begin{definition}\label{defn:sigmas}
\begin{equation} \label{defn:sigmainf}
\sigma_{\infty}^2(\mB,\mS) \eqdef  \min_{v \in \range{\mB^{-1}\mA^\top}}\max_{i=1,\ldots, q} \frac{\norm{ \mB^{1/2}v}_{\mZ_i}^2}{\norm{v}_{\mB}^2}.
\end{equation}
Let $p \in \Delta_q$  and let
\begin{equation} \label{defn:sigmap}
\sigma_{p}^2(\mB,\mS) \eqdef  \min_{v \in \range{\mB^{-1}\mA^\top}} \frac{  \norm{ \mB^{1/2}v}_{\EE{i\sim {p}}{\mZ_i}}^2 }{\norm{v}_{\mB}^2}.
\end{equation}
\end{definition}

Next we show that $\sigma_{\infty}^2(\mB,\mS)$ and $\sigma_{p}^2(\mB,\mS)$ can be used to lower bound $\max_i f_i(x)$ and $\EE{i \sim p}{f_i(x)}$, respectively. This result will allow us to develop~\Cref{eq:onestepprog} and~\Cref{lem:motivation} into a recurrence later on.
  
\begin{lemma} \label{lem:sigmalower} Let $p \in \Delta_q$ and consider the iterates $x^k$ given by~\Cref{alg:adaSKep} when using any adaptive sampling rule.
The spectral constants~\Cref{defn:sigmainf} and~\Cref{defn:sigmap} are such that 
\begin{align}
\max_{i=1,\ldots, q}f_i(x^k)\quad & \geq  \quad
\sigma_{\infty}^2(\mB,\mS) \norm{x^k-x^*}_{\mB}^2, \label{eq:sigmainflower}\\
\EE{i \sim p}{f_i(x^k)} \quad & \geq  \quad\sigma_{p}^2(\mB,\mS)  \norm{x^k-x^*}_{\mB}^2.  \label{eq:sigmaplower}
\end{align}
\end{lemma}
\begin{proof}
From the invariance provided by~\Cref{lem:invariant} we have that $x^k-x^* \in \range{\mB^{-1}\mA^\top}$ and consequently
\begin{align}\label{eq:tempsigminfsd}
 \frac{\max_{i=1,\ldots, q}f_i(x^k)}{\norm{x^k-x^*}_{\mB}^2 } \hspace*{.5em} & \overset{\mathclap{\eqref{eqn:fiBnorm}}}{=}\hspace*{.5em}   \max_{i=1,\ldots, q}\frac{\norm{\mB^{1/2}(x^k-x^*)}_{\mZ_i}^2}{\norm{x^k-x^*}_{\mB}^2} \nonumber\\
&\geq  \min_{v\in \range{\mB^{-1}\mA^\top}}\max_{i=1,\ldots, q}\frac{\norm{\mB^{1/2} v}_{\mZ_i}}{\norm{v}_{\mB}^2}
\;\overset{\eqref{defn:sigmainf}}{=} \;\sigma_{\infty}^2(\mB,\mS), \quad \forall k.
\end{align}

Analogously we have that
\begin{align}
\frac{\EE{i \sim p}{f_i(x^k)}}{\norm{x^k-x^*}_{\mB}^2 } \hspace*{.5em} & \overset{\mathclap{\eqref{eqn:fiBnorm}}}{=} \hspace*{.5em} \frac{\EE{i\sim p}{\norm{\mB^{1/2}(x^k-x^*)}_{\mZ_i}^2}}{\norm{x^k-x^*}_{\mB}^2} \nonumber \\
\hspace*{.5em} & \geq \hspace*{.5em}\min_{v\in \range{\mB^{-1}\mA^\top}}\frac{\EE{i\sim p}{\norm{\mB^{1/2} v}_{\mZ_i}^2}}{\norm{v}_{\mB}^2} \; \overset{\eqref{defn:sigmap}+\eqref{eqn:linNorm}}{=} \; \sigma_{p}^2(\mB,\mS).\label{eq:tempsigmapsd}
\end{align}
Thus~\Cref{eq:sigmainflower} and~\Cref{eq:sigmaplower} follow by re-arranging~\Cref{eq:tempsigminfsd} and~\Cref{eq:tempsigmapsd} respectively.
\end{proof}

Finally, we show that $\sigma_p^2(\mB,\mS) $ and $\sigma_{\infty}^2(\mB,\mS)$ are always less than one, and if the exactness~\Cref{ass:exact} holds then they are both strictly greater than zero. 

\begin{lemma}\label{lem:projEig}
Let $p \in \Delta_q$ and the set of sketching matrices $\{\mS_1,\ldots, \mS_q\}$ be such that 
that exactness~\Cref{ass:exact} holds. 
We then have the following relations: 
\begin{equation*}
0    < \sigma_p^2(\mB,\mS) \quad =\quad  \lambda_{\min}^+ \left( \expSB{i\sim p}{\mZ_{i}}\right) \quad \leq \quad \sigma_{\infty}^2(\mB,\mS) \leq 1.
\end{equation*}
\end{lemma}
\begin{proof}
Using the definition of $\mZ_i$ given in \Cref{eqn:Zk} and the fact that $\mB$ is positive definite, we have 
\begin{eqnarray*}  \label{eq:KerneleqZ}
\kernel{\EE{i \sim p}{\mZ_i}} &\overset{\eqref{eqn:Zk}}{=}& \kernel{\mB^{-1/2}\mA^\top \EE{i \sim p}{\mH_i} \mA \mB^{-1/2}}  \\
& =& \kernel{\mA^\top \EE{i \sim p}{\mH_i} \mA \mB^{-1/2}} \overset{\Cref{lem:NullA}}{=}  \kernel{\mA \mB^{-1/2}},
\end{eqnarray*}
where we applied~\Cref{lem:NullA} in the appendix with $\mG = \EE{i \sim p}{\mH_i}$ and $\mW = \mA.$ Taking the orthogonal complement of the above we have that
\begin{equation}\label{eq:rangeZi}
 \range{\EE{i \sim p}{\mZ_i}} = \range{\mB^{-1/2}\mA^\top}.
\end{equation}
Using the above we then have
\begin{eqnarray*}
\sigma_p^2(\mB,\mS)
&\overset{\eqref{defn:sigmap}}{=}& \min_{v \in \range{\mB^{-1}\mA^\top}} \frac{\norm{  \mB^{1/2} v}_{\expSB{i\sim p}{\mZ_{i}}}^2}{\norm{v}_{\mB}^2} \nonumber\\
 &\overset{\eqref{eq:rangeZi}}{=}&  \min_{\mB^{1/2} v \in \range{\EE{i \sim p}{\mZ_i}}} \frac{\norm{  \mB^{1/2} v}_{\expSB{i\sim p}{\mZ_{i}}}^2}{\norm{v}_{\mB}^2}
\; = \; \lambda_{\min}^{+}\left( \expSB{i\sim p}{\mZ_{i}}\right) >0.\nonumber
 \end{eqnarray*}

 Furthermore, 
\begin{eqnarray*}
\sigma_{p}^2(\mB,\mS) &\overset{\eqref{defn:sigmap}}{=}& \min_{v \in \range{\mB^{-1}\mA^\top}} \frac{\norm{  \mB^{1/2} v}_{\expSB{i\sim p}{\mZ_{i}}}^2}{\norm{v}_{\mB}^2} \nonumber\\
 &\overset{\eqref{eqn:linNorm}}{=}& \min_{v \in \range{\mB^{-1}\mA^\top}} \frac{\expSB{i\sim p}{\norm{  \mB^{1/2} v}_{\mZ_{i}}^2}}{\norm{v}_{\mB}^2} \nonumber\\
&\le&
\min_{v \in \range{\mB^{-1}\mA^\top}}\max_{i=1,\ldots, q} \frac{\norm{  \mB^{1/2} v}_{\mZ_{i}}^2}{\norm{v}_{\mB}^2} 
\;= \;\sigma_{\infty}^2(\mB,\mS) .\nonumber
\end{eqnarray*}
Finally, using the fact that the matrix $\mZ_i$ is an orthogonal projection (\Cref{lem:Z}), we have that 
 \begin{equation*}
\sigma_{\infty}^2(\mB,\mS) \; = \; \max_{i=1,\ldots, q} \frac{\norm{  \mB^{1/2} v}_{\mZ_{i}}^2}{\norm{v}_{\mB}^2} \; \overset{\eqref{eq:Ziproj}}{=} \;
\max_{i=1,\ldots, q} \frac{\norm{ \mZ_i \mB^{1/2} v}^2}{\norm{\mB^{1/2}v}^2} 
\; \le\; \max_{i=1,\ldots, q} \frac{\norm{  \mB^{1/2} v}^2}{\norm{\mB^{1/2}v}^2}  = 1.
 \end{equation*}
\end{proof}

    \subsection{Sampling from a fixed distribution}\label{subsec:fixedSampConv}
We first present a convergence result for the sketch-and-project method when the sketches are drawn from a fixed sampling distribution. This result will later be used as a baseline for comparison against the adaptive sampling strategies. 
\begin{theorem}\label{thm:fixed}
Consider \Cref{alg:fixedProbs} for some set of probabilities $p\in\Delta_q$. 
It follows that
\begin{equation*}
\E{\norm{x^{k}-x^*}_{\mB}^2 }  \leq \left(1-\sigma_p^2(\mB,\mS)\right)^k\norm{x^{0}-x^*}_{\mB}^2.
\end{equation*}
\end{theorem}

\begin{proof} 
Combining \Cref{lem:motivation} and~\Cref{eq:sigmaplower} of \Cref{lem:sigmalower}  we have that
\begin{eqnarray*} 
\EE{i_k \sim p}{\norm{x^{k+1}-x^*}_{\mB}^2 \, | \, x^k }
& \overset{\Cref{lem:motivation}}{=} &  \norm{x^{k}-x^*}_{\mB}^2-  \EE{i_k \sim p}{f_i(x^k)} \nonumber \\
& \overset{ \eqref{eq:sigmaplower}}{\leq} & \left(1-\sigma_p^2(\mB,\mS)\right)  \norm{x^{k}-x^*}_{\mB}^2. \nonumber
\end{eqnarray*}
Taking the full expectation and unrolling the recurrence, we arrive at \Cref{thm:fixed}.
\end{proof}

There are several natural and previously studied choices for fixed sampling distributions, for example, sampling the indices uniformly at random. Another choice is to pick $p\in \Delta_q$ in order to maximize $\sigma_p^2(\mB,\mS)$, but this results in a convex semi-definite program (see Section 5.1 in~\cite{Gower2015} ). The authors of~\cite{Gower2015} suggest convenient probabilities such that $p_i \sim \norm{\mA^\top \mS_i}_{\mB^{-1}}^2$ for which $\sigma_p^2(\mB,\mS)$ reduces to the scaled condition number.

\subsection{Max-distance selection}\label{subsec:max_dist_conv}
The following theorem provides a convergence guarantee for the max-distance selection rule of \Cref{sec:maxdist}.  To our knowledge, this is the first analysis of the max-distance rule for general sketch-and-project methods.

\begin{theorem}\label{thm:maxDistConv}
The iterates of max-distance sketch-and-project method in~\Cref{alg:maxDist} satisfy
\begin{equation*}
\norm{x^{k}-x^*}_{\mB}^2  \leq (1-\sigma_{\infty}^2(\mB,\mS))^k\norm{x^{0}-x^*}_{\mB}^2, 
\end{equation*}
where $\sigma_{\infty}(\mB,\mS)$  is defined as in  \Cref{defn:sigmainf} of \Cref{defn:sigmas}.

\end{theorem}

\begin{proof}

Combining~\Cref{eq:onestepprog} and~\Cref{eq:sigmainflower} we have that
\begin{eqnarray*} 
\norm{x^{k+1}-x^*}_{\mB}^2 
& \overset{\eqref{eq:onestepprog}}{=} &  \norm{x^{k}-x^*}_{\mB}^2-  \max_{i=1,\ldots, q}f_i(x^k) \nonumber \\
& \overset{ \eqref{eq:sigmainflower}}{\leq} & \left(1-\sigma_{\infty}^2(\mB,\mS)\right)  \norm{x^{k}-x^*}_{\mB}^2. \nonumber
\end{eqnarray*}
Unrolling the recurrence gives~\Cref{thm:maxDistConv}.
\end{proof}

One obvious disadvantage of sampling from a fixed distribution is that it is possible to sample the same index twice in a row. Since the current iterate already lies in the solution space with respect to the previous sketch, no progress is made in such an update. For adaptive distributions that only assign non-zero probabilities to non-zero sketched loss values, the same index will never be chosen twice in a row since the sketched loss corresponding to the previous iterate will always be zero (\Cref{lem:rkizero}). This fact allows us to derive convergence rates for adaptive sampling strategies that are strictly better than those for fixed sampling strategies.
\begin{lemma} \label{lem:rkizero}
Consider the sketched losses $f(x^{k})$ generated by iterating the sketch-and-project update given in~\Cref{eqn:xupdate}. We have that
\begin{equation*}
 f_{i_k}(x^{k+1}) =0, \quad \forall \, k\geq 0.
\end{equation*}
\end{lemma}
\begin{proof}

Recall from~\Cref{eqn:fiBnorm}, we can write 
\begin{equation} f_{i_k}(x^{k+1}) = \norm{\mB^{1/2} (x^{k+1}-x^*)}_{\mZ_{i_k}}^2 = \dotprod{\mZ_{i_k}\mB^{1/2} (x^{k+1}-x^*), \mB^{1/2} (x^{k+1}-x^*)}.\label{eq:la39jj9ra}\end{equation}
We can show that the above is equal to zero by using~\Cref{eqn:convstep1} and~\Cref{lem:Z}  we have that
\begin{eqnarray*}
\mZ_{i_k}\mB^{1/2}( x^{k+1} -x^*)& \overset{\eqref{eqn:convstep1}}{=} &  \mZ_{i_k}\mB^{1/2}( x^k - \mB^{-1/2} \mZ_{i_k}\mB^{1/2}( x^k -x^*) -x^*)\nonumber\\
& = &  \mZ_{i_k}\mB^{1/2}( x^k-x^*) - \mZ_{i_k} \mZ_{i_k}\mB^{1/2}( x^k -x^*) )\nonumber\\
& \overset{\eqref{eq:Ziproj} }{=} &  \mZ_{i_k}\mB^{1/2}( x^k-x^*) - \mZ_{i_k}\mB^{1/2}( x^k -x^*) )\\
&=&0.  
\end{eqnarray*}
\end{proof}

We now use~\Cref{lem:rkizero} to additionally show that the convergence guarantee for the greedy method is strictly faster than for sampling with respect to any set of fixed probabilities.

\begin{theorem}\label{thm:maxDistBetterThanUnif}
Let $p \in \Delta_q$ where $ p_i >0$ for all $i = 1,\ldots, q$. Let $\sigma_{p}^2(\mB,\mS)$ be defined as in \Cref{defn:sigmap} of \Cref{defn:sigmas} and define

\begin{equation}\label{eqn:gamma1}
 \gamma \eqdef \frac{1}{\max_{i=1,\ldots,q} \sum_{j=1,\, j \neq i}^q p_j} > 1.
\end{equation}
We then have that the max-distance sketch-and-project method of \Cref{alg:maxDist} satisfies the following convergence guarantee
\begin{equation}\label{eqn:maxDistBetterThanUnif}
\norm{x^{k+1}-x^*}_{\mB}^2  \leq (1-\gamma \sigma_{{p}}^2(\mB,\mS))\norm{x^{k}-x^*}_{\mB}^2.
\end{equation}
\end{theorem}

\begin{proof}
Recall that $f_{i_k}(x^{k+1}) = 0$ by \Cref{lem:rkizero}. Thus,
\begin{align}
    \EE{j\sim {p}}{f_j(x^{k+1})} \hspace*{.5em}&=\hspace*{.5em} \sum_{j=1, \, j \neq i_k}^q {p}_jf_j(x^{k+1})  \nonumber\\
    \hspace*{.5em}&\le\hspace*{.5em} \left(\max_{j=1,\ldots, q}f_j(x^{k+1})\right) \left( \sum_{j=1, \, j \neq i_k}^q {p}_j\right) \nonumber\\
    \hspace*{.5em}&\le\hspace*{.5em} \left(\max_{j=1,\ldots, q}f_j(x^{k+1})\right)\left( \max_{j=1,\ldots, q} \sum_{j=1, \, j \neq i}^q p_j\right) \nonumber\\
    \hspace*{.5em}&\overset{\mathclap{\eqref{eqn:gamma1}}}{=}\hspace*{.5em} \frac{\max_{j=1,\ldots, q}f_j(x^{k+1})}{\gamma}. \label{eqn:MaxDistBetterUnifSimpTheta}
\end{align}
From~\Cref{eq:onestepprog} we have that
\begin{eqnarray*} 
\norm{x^{k+1}-x^*}_{\mB}^2 
& \overset{\eqref{eq:onestepprog}}{=} &  \norm{x^{k}-x^*}_{\mB}^2-  \max_{i=1,\ldots, q}f_i(x^k) \nonumber \\
& \overset{\eqref{eqn:MaxDistBetterUnifSimpTheta}}{\leq} &
\norm{x^{k}-x^*}_{\mB}^2
- \gamma \EE{i\sim p}{f_i(x^{k})} \\
& \overset{\eqref{eq:sigmaplower}}{\leq} &
 \left(1-\gamma \sigma_{p}^2(\mB,\mS)\right)  \norm{x^{k}-x^*}_{\mB}^2. \nonumber
\end{eqnarray*}

\end{proof}

\subsection{The  proportional adaptive rule}\label{subsec:sampPropToRes}

We now consider the adaptive sampling strategy in which indices are sampled with probabilities proportional to the sketched loss values. For this sampling strategy, we derive a convergence rate that is at least twice as fast as that of \Cref{thm:fixed} for uniform sampling.

\begin{theorem}\label{thm:prop_conv}
Consider \Cref{alg:adaSKep} with $p^k = \frac{f(x^k)}{\norm{f(x^k)}_1}$. Let $u = \left(\tfrac{1}{q},\ldots,\tfrac{1}{q} \right) \in \Delta_q$ and $\sigma_{u}^2(\mB,\mS)$ be as defined in \Cref{defn:sigmap}. 
It follows that for $k\ge 1$,
\begin{equation}\label{eq:varconv1step}
\E{\norm{x^{k+1}-x^*}_{\mB}^2 \, | \, x^k}  \leq \left( 1-(1 + q^2 \VAR{i \sim u}{p^k_i})\sigma_{u}^2(\mB,\mS)\right)\norm{x^{k}-x^*}_{\mB}^2, 
\end{equation}
where $\VAR{i \sim u}{\cdot}$ denotes the variance taken with respect to the uniform distribution 
\begin{equation} \label{eq:varU}
\VAR{i \sim u}{v_i} \eqdef \frac{1}{q}\sum_{i=1}^q\left(v_i - \frac{1}{q}\sum_{s=1}^q v_s\right), \quad \forall v \in \R^q.\end{equation}
Furthermore we have that
\begin{equation}
\E{\norm{x^{k+1}-x^*}_{\mB}^2}  \leq \left(1-2\sigma_{u}^2(\mB,\mS)\right)^k \E{\norm{x^{1}-x^*}_{\mB}^2}.
\label{eq:varconvunrolled}
\end{equation}
\end{theorem}

\begin{proof}
First note that for $i \sim u$ we have that
\begin{equation}\label{eqn:var}
    \VAR{u }{f_i(x^k)} = \EE{u} {(f_i(x^k))^2} - \EE{u} {f_i(x^k)}^2\\
    = \frac{1}{q}\sum(f_i(x^k))^2 - \frac{1}{q^2} \left( \sum f_i(x^k)\right)^2.
\end{equation}
Given that $p^k = \frac{f(x^k)}{\norm{f(x^k)}_1}$,  
\begin{align}
   \EE{i\sim p^k}{f_i(x^k)} \hspace*{.5em}&=\hspace*{.5em}  \sum_{i=1}^q p_i^k f_i(x^k) \nonumber\\
   \hspace*{.5em}& =\hspace*{.5em} \sum_{i=1}^q \frac{(f_i(x^k))^2}{\sum_{i=1}^q f_i(x^k)} \nonumber\\
    \hspace*{.5em}& \overset{\mathclap{\eqref{eqn:var}}}{=} \hspace*{.5em} \frac{q\VAR{u }{f_i(x^k)} + \frac{1}{q} \left( \sum f_i(x^k)\right)^2}{\sum_{i=1}^q f_i(x^k) }\nonumber\\
    \hspace*{.5em}&= \hspace*{.5em} \left(q^2\VAR{u }{\frac{f_i(x^k)}{\sum_{i=1}^q f_i(x^k)}} + 1\right)\frac{1}{q}\sum_{i=1}^q f_i(x^k).\label{eqn:nonunif} 
\end{align}
Recalling that $p_i^k = \frac{f_i(x^k)}{\sum_{i=1}^q f_i(x^k)}$ and using~\Cref{lem:motivation} we have that
\begin{equation*} 
\E{\norm{x^{k+1}-x^*}_{\mB}^2 \, | \, x^k }
 \leq   \norm{x^{k}-x^*}_{\mB}^2- (1 + q^2 \VAR{u }{p_i^k})\sigma_{u}^2(\mB,\mS)\norm{x^{k}-x^*}_{\mB}^2.
\end{equation*}

Furthermore, due to \Cref{lem:rkizero} we have that $p^{k+1}_{i_k} =0.$ Therefore
\begin{align*}
\VAR{u }{p_i^{k+1}}& \overset{\eqref{eq:varU}}{=} \frac{1}{q}\sum_{i=1}^q\left(p_i^{k+1} - \frac{1}{q} \sum_{s=1}^q p_s^{k+1}\right)^2\\ 
\hspace*{.5em}&= \hspace*{.5em} \frac{1}{q}\sum_{i=1}^q\left(p_i^{k+1} - \frac{1}{q}\right)^2  \geq  \frac{1}{q}\left(p_{i_k}^{k+1} - \frac{1}{q}\right)^2  = \frac{1}{q^2}.
\end{align*}
This lower bound on the variance gives the following upper bound on~\Cref{eq:varconv1step}
\begin{equation*}
\E{\norm{x^{k+1}-x^*}_{\mB}^2 \, | \, x^k}  \leq  \left(1-2\sigma_{u}^2(\mB,\mS)\right)\norm{x^{k}-x^*}_{\mB}^2. 
\end{equation*}
Taking the expectation and unrolling the recursion  gives~\Cref{eq:varconvunrolled}. 
\end{proof}
Thus by sampling proportional to the sketched losses the sketch-and-project method enjoys a strictly faster convergence rate as compared to sampling uniformly. How much faster depends on the variance of the adaptive probabilities through $1 + q^2 \VAR{u}{p_i^k}$ which in turn depends on the variance of the sketched losses.

 This same variance term is used in \cite{perekrestenko2017fasterCD} to analyze the convergence of an adaptive sampling strategy based on the dual residuals for coordinate descent applied to regularized loss functions and in \cite{osokin16} for adaptive sampling in the block-coordinate Frank-Wolfe algorithm for optimizing structured support vector machines.

\subsection{Capped adaptive sampling}\label{subsec:BaiWuExt}
We now extend the capped adaptive sampling method and convergence guarantees of \cite{BaiWuSISC2018} and \cite{BAIWu201821} for the randomized Kaczmarz setting to the general sketch-and-project setting, see~\Cref{alg:adaSKepUk}. 
Let $p \in \Delta_q$ be a fixed reference probability. 
At each iteration $k$ an index set $\cW_k$ is constructed on line 4 of \Cref{alg:adaSKepUk} that contains indices whose sketched losses are sufficiently close to the maximal sketched loss and that are at least as large as $\EE{i\sim p}{f_i(x^k)}$.  At each iteration, the adaptive probabilities $p_i^k$ are zero for all indices that are not included in the set $\cW_k$. 
The input parameter $\theta \in [0,\,1]$ controls how aggressive the sampling method is. In particular, if $\theta=1$, the method reduces to max-distance sampling. As $\theta$ approaches 0, the sampling method remains adaptive, as only indices corresponding to sketched losses larger than $\EE{i\sim p}{f_i(x^k)}$ are sampled with non-zero probability.
 In \cite{BaiWuSISC2018}, the authors originally introduced an adaptive randomized Kaczmarz method with $\theta = 1/2$. They generalized this in \cite{BAIWu201821} to allow for the more general choice of $\theta\in[0,\,1]$.

\Cref{alg:adaSKepUk} presented here generalizes the method proposed in \cite{BAIWu201821} in three ways. The first is the generalization of the method from the randomized Kaczmarz setting to the more general sketch-and-project setting. The second generalization allows for the use of any fixed reference probability distribution $p \in \Delta_q$, whereas the method of \cite{BaiWuSISC2018} uses sampling proportional to the squared row norms of the matrix $\mA$ as the reference probability. The third generalization is to allow for the use of any adaptive sampling strategy such that the probabilities $p_i^k$ are zero outside of the set $\cW_k$. The methods proposed in \cite{BaiWuSISC2018} and \cite{BAIWu201821} specify that the adaptive probabilities be chosen as $p_{i_k}^k =  f_i(x^k) \mathbf{1}_{i \in \cW_k}/ \sum_{j \in \cW_k} f_j(x^k)$, but this restriction is unnecessary in proving the accompanying convergence result. 

Below, we provide two convergence guarantees for \Cref{alg:adaSKepUk}. \Cref{thm:Bai} provides a convergence guarantee in terms of the spectral constants $\sigma_{\infty}^2(\mB,\mS)$ and  $\sigma_{ p}^2(\mB,\mS)$ of \Cref{defn:sigmas} and the parameter $\theta$. 
\Cref{thm:RGRK} provides a direct generalization of the convergence rate derived in \cite{BAIWu201821}.

\begin{algorithm}
\begin{algorithmic}[1]
\State \textbf{input:}  $x^0\in \R^{n},$ $\mA\in \R^{m\times n}$, $b\in\R^m$,  $p \in \Delta_q$, $\theta\in[0, \,1]$ and a set of sketching matrices $\{\mS_1,\dots, \mS_q\}$
\State \textbf{initialize:} $f_i(x^{0}) = \norm{\mA x^0-b}_{\mH_i} \in \R_+^m$ for $i=1,\ldots, q.$
\For {$k = 0, 1, 2, \dots$}
    \State  $\cW_k = \left\{i \; | \; f_i(x^k) \geq \theta \max_{j=1,\ldots, q} f_j(x^k) + (1-\theta)\EE{j\sim p}{f_j(x^k)}\right\}$
    \State Choose $p^k \in \Delta_q$ such that  $\mbox{support}(p^k)\subset \cW_k$
	\State $i_k\sim p^k$
	\State $x^{k+1} = x^k - \mB^{-1}\mA^\top \mH_{i_k}(\mA x^k-b)$
	\label{ln:xupdate22}
	\State update $f_i(x^{k+1}) = \norm{\mA x^{k+1}-b}_{\mH_i}$ for $i=1,\ldots, q.$
\EndFor
\State \textbf{output:} last iterate $x^{k+1}$
\end{algorithmic}
\caption{Capped Adaptive Sketch-and-Project}
\label{alg:adaSKepUk}
\end{algorithm}

\begin{theorem} \label{thm:Bai}
Consider \Cref{alg:adaSKepUk}.
Let $p \in \Delta_q$ be a  fixed \emph{reference probability} and $\theta\in [0,1]$.
Let 
\begin{equation}
    \cW_k = \left\{i \; | \; f_i(x^k) \geq \theta \max_{j=1,\ldots, q} f_j(x^k) + (1-\theta)\EE{j\sim p}{f_j(x^k)}\right\}. \label{eq:Ukdefgamma}
\end{equation}
 It follows that
\begin{equation}\label{eq:o9s88js84js4}
\E{\norm{x^{k}-x^*}_{\mB}^2}  \leq \left(1-\theta \sigma_{\infty}^2(\mB,\mS) - (1-\theta)\sigma_{p}^2(\mB,\mS)\right)^k\norm{x^{0}-x^*}_{\mB}^2.
\end{equation}
\end{theorem}

\begin{proof}
First note that $\cW_k$ is not empty since 
\[ \max_{j=1,\ldots, q} f_j(x^k) \geq  \EE{j\sim p}{f_j(x^k)},\]
and thus $\arg\max_{j=1,\ldots, q} f_j(x^k) \in \cW_k.$
Since $p_i^k = 0$ for all $i\not\in\cW_k$,~\Cref{lem:motivation} gives that
\begin{equation}
\EE{i\sim p^k}{\norm{x^{k+1}-x^*}_{\mB}^2 \, | \, x^k }
= \norm{x^{k+1}-x^*}_{\mB}^2- \sum_{i\in \cW_k} p_i^k f_i(x^k).\label{eqn:tempj8j8s44s}
\end{equation}
We additionally have 
\begin{align}
\sum_{i \in \cW_k} f_i(x^k) p_i^k 
\hspace*{1.5em}&\overset{\mathclap{\eqref{eq:Ukdefgamma}}}{\geq}  \hspace*{1.5em}
\sum_{i \in \cW_k}\left( \theta \max_{j=1,\ldots, q} f_j(x^k) + (1-\theta)\EE{j\sim p}{f_j(x^k)}\right) p_i^k \nonumber\\
\hspace*{1.5em}& =\hspace*{1.5em}  \theta \max_{j=1,\ldots, q} f_j(x^k) + (1-\theta)\EE{j\sim p}{f_j(x^k)}  \label{eqn:needlaterj83j}\\ 
\hspace*{1.5em}&\overset{\mathclap{\Cref{lem:sigmalower}}}{\geq}\hspace*{1.5em}  \left(\theta \sigma_{\infty}^2(\mB,\mS) +
(1-\theta)\sigma_{p}^2(\mB,\mS) \right) \norm{x^k-x^*}_{\mB}^2. \label{eq:j8j8aj3}
\end{align}
Using~\Cref{eq:j8j8aj3} to bound~\Cref{eqn:tempj8j8s44s} and taking the expectation gives the result.
\end{proof}

The resulting convergence rate is a convex combination of the spectral constant $\sigma_{\infty}^2(\mB,\mS)$ which corresponds to the max-distance convergence rate guarantee and $\sigma_{ p}^2(\mB,\mS)$ corresponding to the convergence rate guarantee for the fixed reference probabilities $ p$. This convex combination is in terms of the parameter $\theta$ and we can see that as $\theta$ approaches 1 the method and convergence guarantee approach that of max-distance. When $\theta$ is close to 0, the convergence guarantee approaches that of a fixed distribution, but still filters out sketches with sketched losses less than $\EE{j\sim p}{f_j(x^k)}$. This suggests that for $\theta\approx 0$ the convergence rate guarantee is loose.

We now explicitly extend the analysis of Bai and Wu's work of \cite{BAIWu201821} to derive a convergence rate guarantee for our more general \Cref{alg:adaSKepUk}.

\begin{theorem}\label{thm:RGRK}
Consider \Cref{alg:adaSKepUk}.
Let $p\in\Delta_q$ be a set of fixed \emph{reference probabilities} and $\theta\in [0,1]$.
Let 
\begin{equation}\label{eq:theta}
 \gamma \eqdef \frac{1}{\max_{i=1,\ldots,q} \sum_{j=1,\, j \neq i}^q p_j} > 1.
\end{equation}
 It follows for $k \geq 1$
\begin{align}\label{eqn:BaiWuBoundExt}
&\E{\norm{x^{k}-x^*}_{\mB}^2} \\
&\leq \left(1-\left(\theta\gamma +(1-\theta)\right)\sigma_{p}^2(\mB,\mS)\right)^{k-1} \left(1-\theta \sigma_{\infty}^2(\mB,\mS) - (1-\theta)\sigma_{p}^2(\mB,\mS)\right)\norm{x^{0}-x^*}_{\mB}^2,\nonumber
\end{align}
where the expectation is taken with respect to the probabilities prescribed by Algorithm~\ref{alg:adaSKepUk}.
\end{theorem}

\begin{proof}
By \Cref{lem:rkizero}, at least one of the sketched losses is guaranteed to be zero for each iterations $k\geq 1$. Making the conservative assumption that this sketched loss corresponds to the smallest probability $\hat p_{i_k}^k$, we have, by \Cref{eqn:MaxDistBetterUnifSimpTheta}, that for an adaptive sampling strategy that assigns $p_i^k=0$ to sketches $\mS_i$ with a sketched loss $f_i(x^k)=0$ that
\begin{equation}
\displaystyle \frac{\max_{j=1,\ldots, q} f_j(x^{k+1})}{\EE{j\sim p}{f_j(x^{k+1})}} \geq \gamma. \label{eq:thetaineq}
 \end{equation}
 
Combining this with \Cref{eqn:needlaterj83j},
\begin{align}
\sum_{i \in \cW_k} f_i(x^{k+1}) p_i^{k+1}\hspace*{.5em} & \geq \hspace*{.5em} \left(\theta \frac{\max_{j=1,\ldots, q} f_j(x^{k+1})}{\EE{j\sim p}{f_j(x^{k+1})}} + (1-\theta)\right)\EE{j\sim p}{f_j(x^{k+1})}  \nonumber \\
\hspace*{.5em}& \overset{\mathclap{\eqref{eq:thetaineq}}}{\geq}\hspace*{.5em}  \left(\theta \gamma+ (1-\theta)\right)\EE{j\sim p}{f_j(x^{k+1})}  \nonumber \\
\hspace*{.5em}& \overset{\mathclap{\eqref{defn:sigmap}}}{\geq} \hspace*{.5em} \left(\theta \gamma+ (1-\theta)\right)\sigma_{p}^2(\mB,\mS).\label{eqn:kgr1bound}
\end{align}
Consequently for $k\ge 1$, by \Cref{eqn:tempj8j8s44s}, we then have
\[
\E{\norm{x^{k+1}-x^*}_{\mB}^2 \, | \, x^k } \le \norm{x^k-x^*}_{\mB}^2- \left(\theta \gamma+ (1-\theta)\right)\sigma_{p}^2(\mB,\mS) \norm{x^k-x^*}_{\mB}^2.
\]
Taking the expectation and unrolling the recursion gives,
\[
\E{\norm{x^{k+1}-x^*}_{\mB}^2 } \le \left(1-\left(\theta \gamma+ (1-\theta)\right)\sigma_{p}^2(\mB,\mS) \right)^{k-1} \norm{x^1-x^*}_{\mB}^2.
\]
Since, at the very first update, we cannot guarantee that there exists $i\in [1,\ldots, q]$ such that $f_i(x^{0}) = 0$, \Cref{eqn:kgr1bound} is not guaranteed for $k=0$. So instead we use~\Cref{eq:o9s88js84js4} to unroll the last step in this recurrence to arrive \Cref{eqn:BaiWuBoundExt}.
\end{proof}

The convergence rate for \Cref{alg:adaSKepUk} of \Cref{thm:RGRK} is an improvement over the convergence rate guarantee for a fixed probability distribution since $\gamma>1$. 
As was the case for \Cref{thm:Bai}, the convergence rate is maximized when $\theta = 1$, at which point the resulting  method is equivalent to the max-distance sampling strategy of \Cref{alg:maxDist}. Further, when $\theta =1$, \Cref{thm:RGRK} guarantees
\begin{equation*}
\E{\norm{x^{k}-x^*}_{\mB}^2}  \leq \left(1-\gamma\sigma_{p}^2(\mB,\mS) \right)^{k-1}\left(1-\sigma_{p}^2(\mB,\mS)\right)\norm{x^{0}-x^*}_{\mB}^2.
\end{equation*}
For $\theta = 0$, \Cref{thm:RGRK} recovers the same convergence guarantee as for sampling according to the non-adaptive probabilities $p$.

\begin{table}[]
    \centering
    \begin{tabular}{|c|c|c|c|c|}
    \hline
        \textbf{\makecell{Sampling\\Strategy}} & \textbf{\makecell{Convergence \\Rate Bound}} & \textbf{\makecell{Rate Bound\\Shown In}} 
        \\        \hline \hline
        Fixed, $p_i^k \equiv p_i$ & $1-\sigma_p^2(\mB,\mS)$  &  \cite{Gower2015},  \Cref{thm:fixed} 
        \\ \hline
        Max-distance & $1-\sigma_\infty^2(\mB,\mS)$ &  \Cref{thm:maxDistConv} 
        \\ \hline
        $p_i^k\propto f_i(x^k)$ & $1-2\sigma_{u}^2(\mB,\mS)$ & \Cref{thm:prop_conv}
        \\ \hline
        Capped & $1 - \left(1 + \epsilon\right)\sigma_{ p}^2(\mB,\mS)$ &  \Cref{thm:RGRK} 
        \\ \hline
    \end{tabular}
    \caption{Summary of convergence guarantees of \Cref{sec:convergence}, where $\gamma= 1/ \underset{i=1,\ldots, m}{\max}\sum_{j=1, j\ne i}^m p_i$ as defined in \Cref{eqn:gamma1} and $\epsilon = \theta(\gamma-1) \le \theta \tfrac{1}{m}$.}
    \label{tab:conv_summary}
\end{table}

\section{Implementation tricks and computational complexity}
\label{sec:imp_tricks}

One can perform adaptive sketching with the same order of cost per iteration as the standard non-adaptive sketch-and-project method when $\tau q$, the number of sketches $q$ times the sketch size $\tau$, is not significantly larger than the number of columns $n$. In particular, adaptive sketching methods can be performed for a per-iteration cost of $O(\tau^2 q + \tau n)$, whereas the standard non-adaptive sketch-and-project method has a per-iteration cost of $O( \tau n)$. 
The main computational costs of adaptive sketch-and-project (\Cref{alg:adaSKep}) at each iteration come from computing the sketched losses $f_i(x^k)$ of \Cref{eqn:residualk} and updating the iterate from $x^k$ to $x^{k+1}$ via \Cref{eqn:xupdate}. 
The iterate update for $x^k$ and the formula for the sketched loss $f_i(x^k)$ both require calculating what we call the \emph{sketched residual}, 
\begin{equation}\label{def:sketched-res}
    \mR_i^k \eqdef  \mC_i^\top \mS_i^\top( \mA x^k - b),
\end{equation}
where $\mC_i$ is any square matrix satisfying 
$    \mC_i \mC_i^\top = (\mS_{i}^\top \mA \mB^{-1} \mA^\top \mS_{i})^\dagger.$ 
The adaptive methods considered here require the sketched residual $\mR_i^k$ for each sketch index $i=1, 2, \ldots, q$ at each iteration. For such adaptive methods, it is possible to update the iterate $x^k$ and compute the sketched losses $f_i(x^k)$ more efficiently if one maintains the set of sketched residuals
$\{\mR_i^k : i = 1, 2, \ldots, q\}$ in memory. 
\Cref{sec:imp_tricks_full} discusses the costs of adaptive sketch-and-project methods in more detail. Pseudocode for efficient implementation is provided in \Cref{alg:adasketch}.

Different sampling strategies require different amounts of computation as well. Among the adaptive sampling strategies considered here, max-distance sampling requires the least amount of computation followed by sampling proportional to the sketched losses. Capped adaptive sampling requires the most computation. The costs for each sampling strategy are discussed in detail in  \Cref{sec:sample_spec_costs} and are summarized in \Cref{tab:general-rule-costs}.

\section{Summary of consequences for special cases}\label{sec:costs-and-convergence-summaries}

We now discuss the consequences of the convergence analyses of \Cref{sec:convergence} and the computational costs detailed in \Cref{sec:imp_tricks} for the special sketch-and-project subcases of randomized Kaczmarz and coordinate descent. For $\mC_i$ as defined in \Cref{def:C_i}, in both the randomized Kaczmarz method and coordinate descent, $\mC_i$ is a scalar and thus its value is fixed.

\subsection{Adaptive Kaczmarz}\label{subsec:adakacz-summary}

By choosing the parameter matrix $\mB=\mI$ and sketching matrices $\mS_i = e_i$ for $i = 1,\ldots,m$ where $e_i\in \R^n$ is the $i\thup$ coordinate vector, we arrive at the Kaczmarz method introduced in \Cref{subsec:RK_intro}. For randomized Kaczmarz, the sketches $\mS_i = e_i$ isolate a single row of the matrix $\mA$, as $\mS_i^\top \mA = \mA_{i:}$. In this setting, the number of sketches $q = m$ for $\mA\in \R^m$, and the sketch size is $\tau = 1$. In order to perform the adaptive update efficiently, the matrices
\[
\mB^{-1}\mA^\top \mS_{i}\mC_{i} = \frac{\mA_{i:}^\top}{\norm{\mA_{i:}}} 
\; \mbox{ and }\;
\mC_i^\top\mS_i^\top \mA \mB^{-1}\mA^\top \mS_{j}\mC_{j} = \frac{\ve{\mA_{i:}}{\mA_{j:}}}{\norm{\mA_{i:}}\norm{\mA_{j:}}} \; \forall\, i, j = 1, 2, \ldots m
\] 
should be precomputed.

In order to succinctly express the convergence rates, we define the diagonal probability matrix $\mP= \diag(p_1,\ldots,p_m)$ and the normalized matrix $\bar \mA \eqdef \mD_{RK}^{-1}\mA$, with $\mD_{RK} \eqdef \diag\left(\norm{\mA_{1:}}_2,
\ldots,\norm{\mA_{m:}}_2\right)$ as in \cite{Nutini2016}. In the randomized Kaczmarz setting, the projection matrix $\mZ_i$ as defined in \Cref{eqn:Zk} is the orthogonal projection onto the $i\thup$ row of $\mA$ and takes the form 
\[
\mZ_i = \frac{\mA_{i:}\mA_{i:}^\top}{\norm{\mA_{i:}^2}}.
\]
We then have 
\[
\EE{i\sim p}{\mZ_i} = \mD_{RK}^{-1} \mA \mP \mA^\top \mD_{RK}^{-1} = \bar \mA^\top \mP \bar \mA.
\]
The costs and convergence rates for the adaptive sampling strategies discussed in \Cref{sec:methods} applied to the Kaczmarz method are summarized in \Cref{tab:adakacz}, where we used the notation $\norm{x}_{\infty} \eqdef \max_i |x_i|$ for any vector $x$.

\begin{table}[]
\renewcommand*{\arraystretch}{1.25}
    \centering
\resizebox{\columnwidth}{!}{%
    \begin{tabular}{|c|c|c|c|}
    \hline
        \makecell{\textbf{Sampling}\\ \textbf{Strategy}} & \makecell{\textbf{Convergence}\\ \textbf{Rate Bound}} & \makecell{\textbf{Rate Bound}\\ \textbf{Shown In}} & \makecell{\textbf{Flops Per}\\ \textbf{Iteration}}
        \\        \hline \hline
        Uniform  & $1-\frac{1}{m}\lambda_{\min}^+(\bar \mA^\top \bar \mA)$  & \cite{Nutini2016}, \Cref{thm:fixed} & $2 \min(n,m) + 2n$\\ \hline
        $p_i\propto \norm{\mA_{i:}}_2^2$  & $1 - \frac{\lambda_{\min}^+( \mA^\top \mA)}{\norm{\mA}_F^2}$ & \cite{Strohmer2009}, \Cref{thm:fixed} & $2 \min(n,m) + 2n$\\ \hline
        Max-distance & $1-\underset{v \in \range{\mA^\top}}{\min} \frac{ \norm{\bar\mA v}_\infty}{\norm{v}_2}$ & \cite{Nutini2016}, \Cref{thm:maxDistConv} & $3 m + 2 n$ \\ \hline
        $p_i^k\propto f_i(x^k)$ & $1-\frac{2}{m}\lambda_{\min}^+(\bar \mA^\top \bar \mA)$ & \Cref{thm:prop_conv} & $5 m + 2 n$ \\ \hline
        Capped
& $1 - \left(\theta\gamma + 1\right)\lambda_{\min}^+(\bar \mA^\top \mP \bar \mA)$ & \cite{BAIWu201821}, \Cref{thm:RGRK} &  $9 m + 2 n$ \\ \hline
    \end{tabular}
}
    \caption{Summary of convergence guarantees and costs of various sampling strategies for the randomized Kaczmarz algorithm. Here, $\gamma= 1/ \underset{i=1,\ldots, m}{\max}\sum_{j=1, j\ne i}^m p_i$ as defined in~\Cref{eqn:gamma1}, $\mP= \diag(p_1,\ldots,p_m)$ is a matrix of arbitrary fixed probabilities, and $\bar \mA \eqdef  \mD_{RK}^{-1}\mA$, with $\mD_{RK} \eqdef \diag\left(\norm{\mA_{1:}}_2,
\ldots,\norm{\mA_{m:}}_2\right)$. Only leading order flop counts are reported. The number of sketches is $q$, the sketch size is $\tau$ and the number of rows and columns in the matrix $\mA$ are $m$ and $n$ respectively.}
    \label{tab:adakacz}
\end{table}

\subsection{Adaptive coordinate descent}\label{subsec:adacd-summary}
By choosing the parameter matrix $\mB=\mA^\top \mA$ and sketching matrices $\mS_i =\mA e_i$ for $i = 1,\ldots,n$ where $e_i \in \R^m$ is the $i\thup$ coordinate vector, we arrive at the coordinate descent method introduced in \Cref{subsec:CD_intro}. In this setting, the number of sketches $q = n$, where $n$ is number of columns in $\mA$, and the sketch size is $\tau = 1$.

Coordinate descent uses fewer flops per iteration than indicated by the general computation given in \Cref{subsec:per_iter_cost}. This computational savings arises from the sparsity of the matrix $\mB^{-1}\mA^\top \mS_{i_k}\mC_{i_k} = e_i / \norm{\mA_{:i}}$. As a result, the iterate update of $x^k$ to $x^{k+1}$ using the sketched residuals $\mR_{i_k}^k$ requires only $O(1)$ flops instead of $2 n$ flops as indicated in the general analysis that is summarized in \Cref{tab:general-shared-costs}. The cost of a coordinate descent update is dominated by the $2n$ flops required to calculate $\mR_{i_k}^k$ by either the auxiliary update of \Cref{ln:efficient-auxiliary-update} or directly via \Cref{def:sketched-res}.

Similar to the randomized Kaczmarz case, we define the diagonal probability matrix $\mP \eqdef  \diag(p_1,\ldots,p_n)$ and the normalized matrix $\widetilde \mA \eqdef \mA \mD_{CD}^{-1}$, with $\mD_{CD}\eqdef \diag\left(\norm{\mA_{:1}}_2,
\ldots,\norm{\mA_{:n}}_2\right)$. The projection matrix $\mZ_i$ as defined in \Cref{eqn:Zk} is the projection given by
\[
\mZ_i = (\mA^\top \mA)^{-1/2}\mA^\top\mA \frac{e_i e_i^\top}{\norm{\mA_{:i}}^2} \mA^\top\mA (\mA^\top \mA)^{-1/2}
= (\mA^\top \mA)^{1/2} \frac{e_i e_i^\top}{\norm{\mA_{:i}}^2}  (\mA^\top \mA)^{1/2}.
\]
We then have 
\[
\EE{i\sim p}{\mZ_i} = (\mA^\top \mA)^{1/2} \mD_{CD}^{-1} \mP \mD_{CD}^{-1}  (\mA^\top \mA)^{1/2}.
\]
Note that $\EE{i\sim p}{\mZ_i} $ is similar to $ \mP \mD_{CD}^{-1}  \mA^\top \mA \mD_{CD}^{-1} = \mP \widetilde \mA^\top \widetilde \mA$ and thus 
\[
\lambda_{\min}^+(\EE{i\sim p}{\mZ_i}) = \lambda_{\min}^+ (\mP \widetilde \mA^\top \widetilde \mA).
\]
The costs and convergence rates for the adaptive sampling strategies discussed in \Cref{sec:methods} applied to coordinate descent are summarized in \Cref{tab:adacd}.

\begin{table}[]
    \centering
    \begin{tabular}{|c|c|c|c|}
    \hline
        \textbf{Sampling} & \makecell{\textbf{Convergence}\\ \textbf{Rate Bound}} & \makecell{\textbf{Rate Bound}\\ \textbf{Shown In}} &
        \makecell{\textbf{Flops Per}\\ \textbf{Iteration}}
        \\        \hline \hline
        Uniform & $1-\frac{1}{n}\lambda_{\min}^+(\widetilde \mA^\top \widetilde \mA)$ & \Cref{thm:fixed} & 
        $2n$ \\ \hline
        $p_i\propto \norm{\mA_{:i}}_2^2$ & $\left(1-\frac{\lambda_{\min}^+(\mA^\top \mA)}{\norm{\mA}_F^2}\right)$ & \cite{leventhal2010randomized} \Cref{thm:fixed} & 
        $2n$  \\ \hline
        Max-distance & $1-\underset{v \in \range{\mA^\top}}{\min} \frac{ \norm{\widetilde\mA v}_\infty}{\norm{v}_2}.$ & \Cref{thm:maxDistConv} & 
        $3n$\\ \hline
        $p_i^k \propto f_i(x^k)$ & $1-\frac{2}{n}\lambda_{\min}^+(\widetilde \mA^\top \widetilde \mA)$ & \Cref{thm:prop_conv} & 
        $5n$ \\ \hline
        Capped & $1 - \left(\theta\gamma + 1\right)\lambda_{\min}^+(\mP\widetilde \mA^\top  \widetilde \mA)$ & \Cref{thm:RGRK} & 
        $9n$ \\ \hline
    \end{tabular}
    \caption{Summary of convergence guarantees and costs of various sampling strategies for adaptive coordinate descent. Here, $\gamma= 1/ \underset{i=1,\ldots, m}{\max}\sum_{j=1, j\ne i}^m p_i$ as defined in~\Cref{eqn:gamma1}, $\mP= \diag(p_1,\ldots,p_m)$ is a matrix of arbitrary fixed probabilities, and $\widetilde \mA = \mA\mD_{CD}^{-1}$, with $\mD_{CD} = \diag\left(\norm{\mA_{:1}}_2,
\ldots,\norm{\mA_{:n}}_2\right)$. Only flop counts of leading order are reported.
}
    \label{tab:adacd}
\end{table}

\section{Experiments}\label{sec:experiments}
We test the performance of various adaptive and non-adaptive sampling strategies in the special sketch-and-project subcases of randomized Kaczmarz and coordinate descent. We report performance via three different metrics: norm-squared error versus iteration, norm-squared error versus approximate flop count, and the worst expected convergence factor.

Results are averaged over 50 trials. For each trial a single matrix $\mA$ is used. For the experiments measuring error, a single true solution $x^*$ and vector $b$ are used. 
To find the worst expected convergence factor, a new exact solution $x^*$ is generated for each trial, since the max-distance method is deterministic and this adds more variation between trials. The exact solutions $x^*$ are generated by 
\begin{equation*}
    x^* = \frac{\mA^\top \omega}{\norm{\mA^\top \omega}_{\mB}},
\end{equation*}
where $\omega\in \R^m$ is a vector of i.i.d. random normal entries. 
 Thus $\norm{x^*}_{\mB}^2 =1$ is normalized with respect to the $\mB$--norm and lies in the row space of $\mA$. The latter condition guarantees that $x^*$ is indeed the unique solution to \Cref{eqn:problin}. 
We measure the error in terms of the $\mB$-norm. Recall that for randomized Kaczmarz $\mB =\mI$ , while for coordinate descent, $\mB = \mA^\top \mA$. The sketch-and-project methods are implemented using the auxiliary update \Cref{ln:efficient-auxiliary-update} as detailed in \Cref{alg:adasketch}.


 We consider synthetic matrices of size $1000\times100$ and $100 \times 1000$ that are generated with i.i.d. standard Gaussian entries.  
We additionally test the various adaptive sampling strategies on two large-scale matrices arising from real world problems. These matrices are available via the SuiteSparse Matrix Collection \cite{suitesparse}. The first system (Ash958) is an overdetermined matrix with 958 rows, 292 columns, and 1916 entries \cite{duff1989sparse,duff1992users}. The matrix comes from a survey of the United Kingdom and is part of the original Harwell sparse matrix test collection. 
The second real matrix we consider is the GEMAT1 matrix, which arises from optimal power flow modeling. This matrix is highly underdetermined and consists of 4929 rows, 10,595 columns, and 47,369 entries \cite{duff1989sparse,duff1992users}.

\subsection{Error per iteration}
We first investigate the convergence of the squared norm of the error, $\norm{x^k - x^*}_{\mB}^2$ in terms of the number of iterations, see~\Cref{fig:synth-err-vs-iter}. 
The first row of subfigures (\Cref{subfig:synthetic_errsq-vs-iter-100x1000-rk,subfig:synthetic_errsq-vs-iter-1000x100-rk}) shows convergence for randomized Kaczmarz, while the second row of subfigures (\Cref{subfig:synthetic_errsq-vs-iter-100x1000-cd,subfig:synthetic_errsq-vs-iter-1000x100-cd}) gives the convergence of various sampling strategies for coordinate descent. The first column of subfigures (\Cref{subfig:synthetic_errsq-vs-iter-100x1000-rk,subfig:synthetic_errsq-vs-iter-100x1000-cd}) uses an underdetermined system of $100\times 1000$ while the second column of subfigures (\Cref{subfig:synthetic_errsq-vs-iter-1000x100-rk,subfig:synthetic_errsq-vs-iter-1000x100-cd}) considers an overdetermined system of $1000\times 100$. \Cref{ash958_flops-vs-iter-958x292-rk,ash958_flops-vs-iter-958x292-cd} demonstrate convergence per iteration for the Ash958 matrix and \Cref{gemat1_errsq-vs-iter-4929x10595-rk,gemat1_errsq-vs-iter-4929x10595-cd} for randomized Kaczmarz and coordinate descent applied to the GEMAT1 matrix.

As expected, we see that the max-distance sampling strategy performs the best per iteration followed by the capped adaptive strategy, then sampling proportional to the sketched residuals and finally followed by the uniform strategy. For randomized Kaczmarz applied to underdetermined systems and coordinate descent applied to overdetermined systems, max-distance and the capped adaptive sampling strategies perform similarly in terms of squared error per iteration. The convergence of randomized Kaczmarz for each sampling strategy applied to overdetermined systems is very similar to that of coordinate descent applied to underdetermined systems. Similarly, the convergence of randomized Kaczmarz for each sampling strategy applied to underdetermined systems is very similar to that of coordinate descent applied to overdetermined systems. 
For the large and underdetermined GEMAT1 matrix, we find that randoimized coordinate descent methods have much larger variance in their performance compared to randomized Kaczmarz methods.

\begin{figure}[h]
    \begin{subfigure}{0.48\textwidth}
        \includegraphics[width=\textwidth]{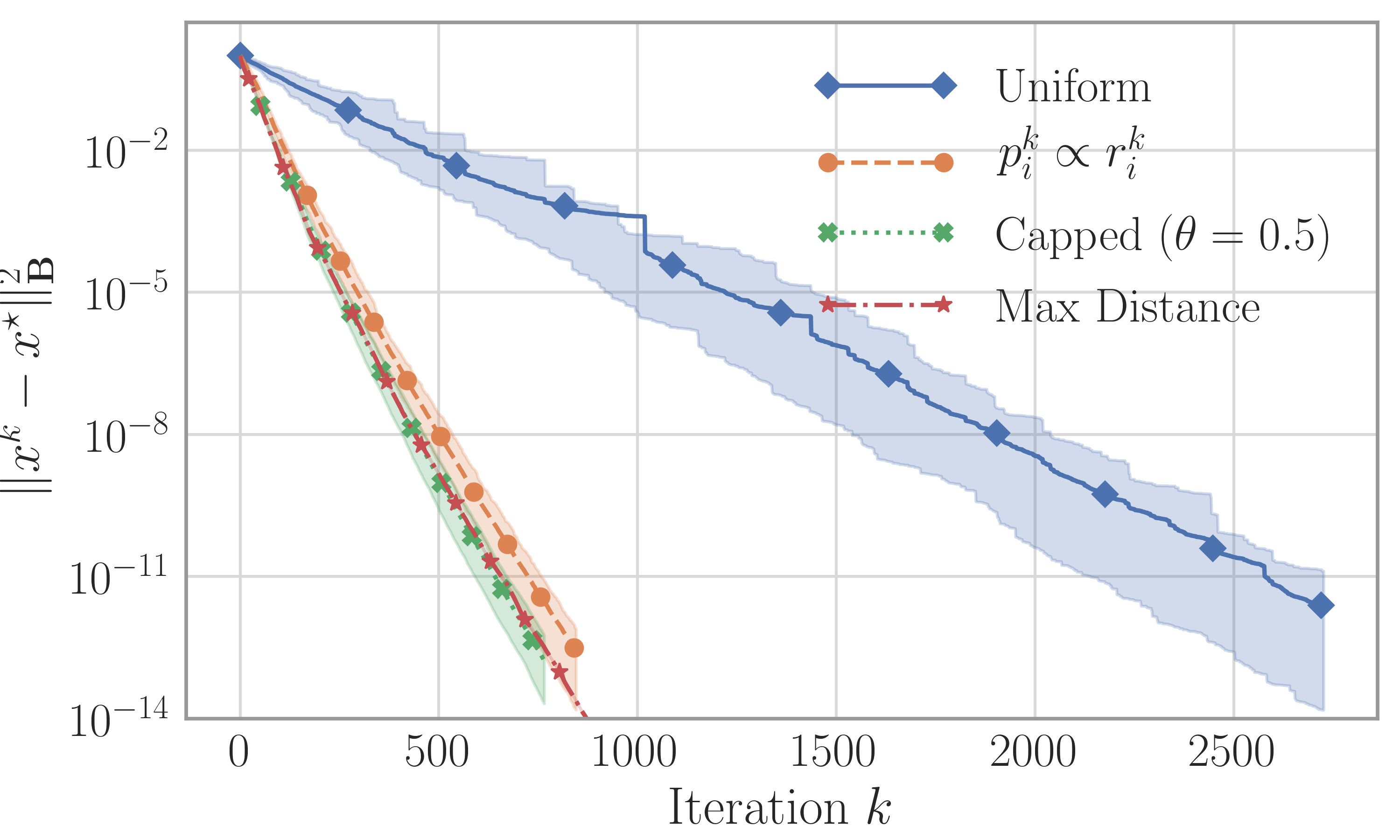}
        \caption{Adaptive randomized Kaczmarz, $\mA \in \R^{100 \times 1000}$.}
        \label{subfig:synthetic_errsq-vs-iter-100x1000-rk}
    \end{subfigure}
    \hspace{0.01\textwidth} 
    \begin{subfigure}{0.48\textwidth}
        \includegraphics[width=\textwidth]{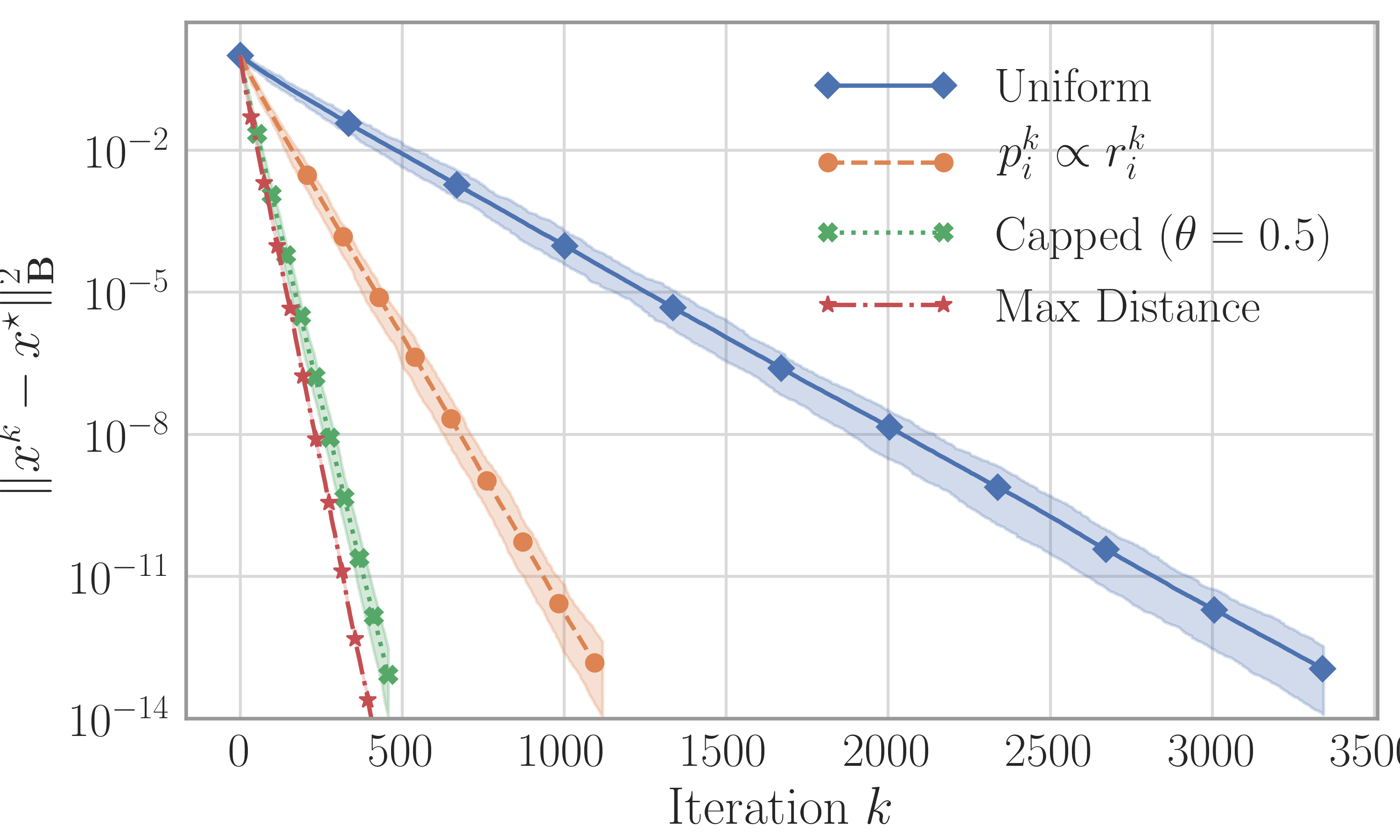}
        \caption{Adaptive randomized Kaczmarz, $\mA \in \R^{1000 \times 100}$.}
        \label{subfig:synthetic_errsq-vs-iter-1000x100-rk}
    \end{subfigure}
    \begin{subfigure}{0.48\textwidth}
        \includegraphics[width=\textwidth]{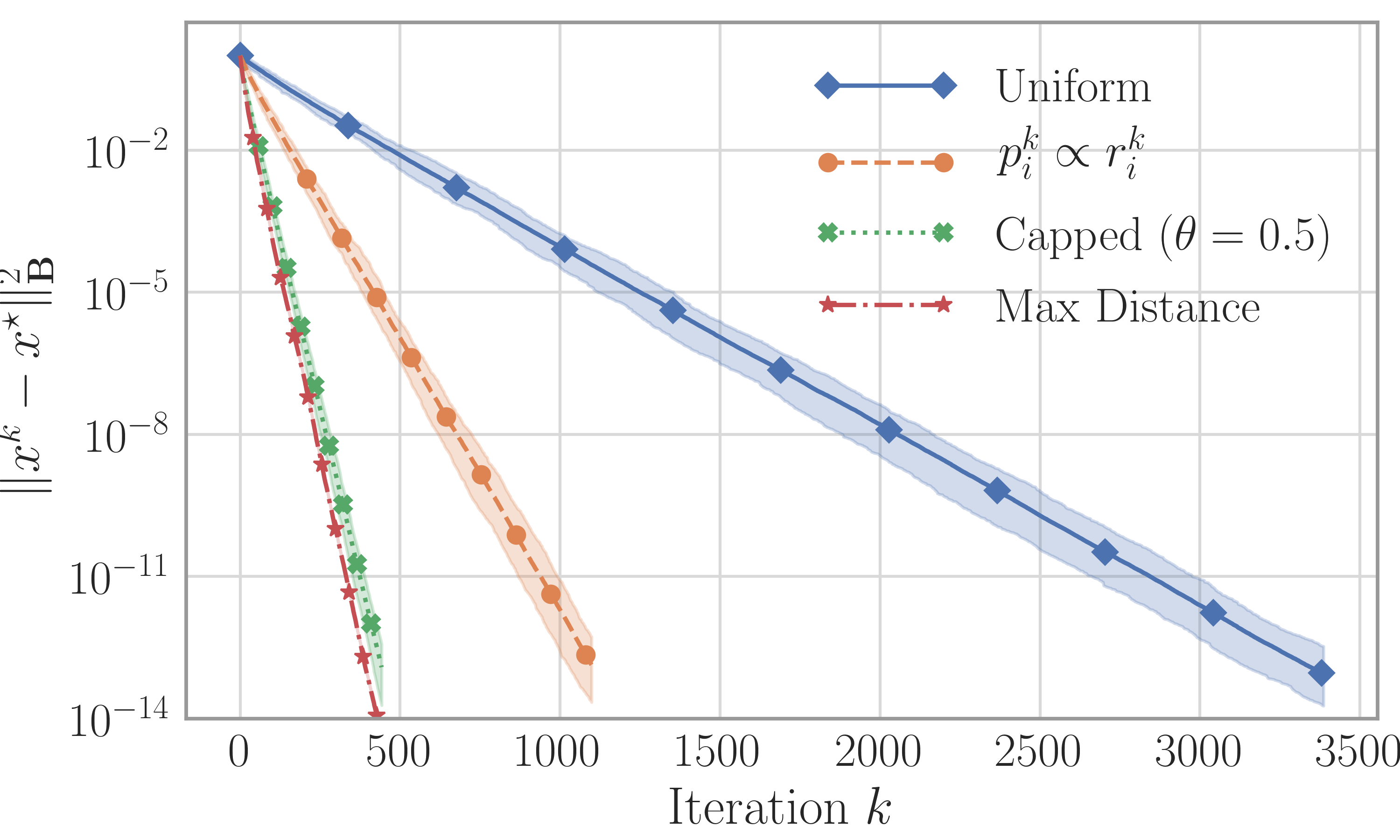}
        \caption{Adaptive coordinate descent, $\mA \in \R^{100 \times 1000}$.}
        \label{subfig:synthetic_errsq-vs-iter-100x1000-cd}
    \end{subfigure}
    \hspace{0.01\textwidth}
    \begin{subfigure}{0.48\textwidth}
        \includegraphics[width=\textwidth]{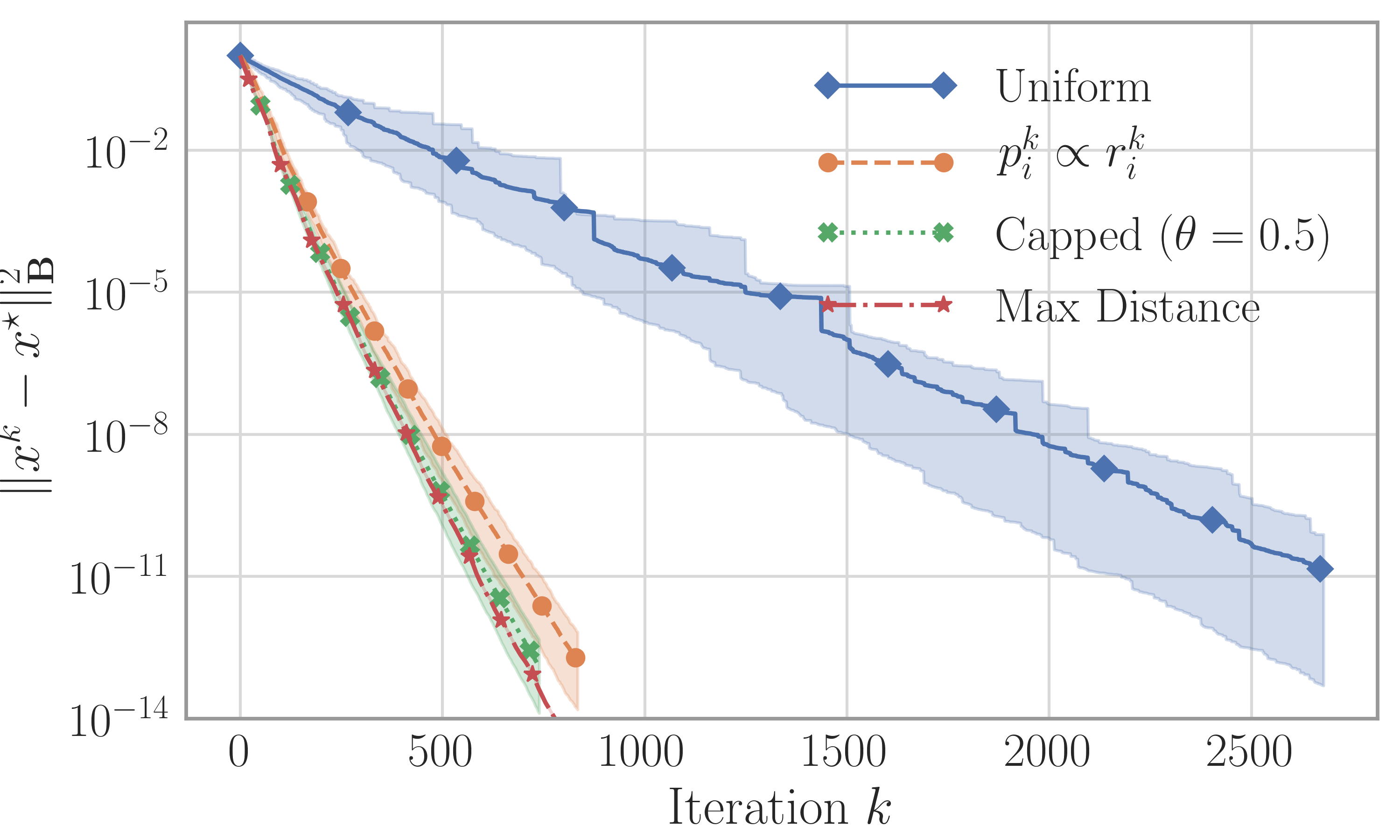}
        \caption{Adaptive coordinate descent, $\mA \in \R^{1000 \times 100}$.}
        \label{subfig:synthetic_errsq-vs-iter-1000x100-cd}
    \end{subfigure}
\centering
\caption{A comparison between different selection strategies for randomized Kaczmarz and coordinate descent methods. Squared error norms were averaged over 50 trials. Confidence intervals indicate the middle 95\% performance. Subplots on the left show convergence for underdetermined systems, while those on the right show the convergence on an overdetermined systems.}
\label{fig:synth-err-vs-iter}
\end{figure}

\subsection{Error versus approximate flops required}
If we take into account the number of flops required for each method, the relative performance of the methods changes significantly. In order to approximate the number of flops required for each sampling strategy, we use the leading order flop counts per iteration given in \Cref{tab:adakacz,tab:adacd}. We do not consider the precomputational costs, but only the costs incurred at each iteration. The performance in terms of flops of each sampling strategy is reported in \Cref{fig:synth-err-vs-flops}. Performance on the Ash958 matrix is reported in \Cref{ash958_flops-vs-iter-958x292-rk,ash958_flops-vs-iter-958x292-cd}. Performance on the GEMAT1 matrix for randomized Kaczmarz and coordinate descent is reported in \Cref{gemat1_flops-vs-iter-4929x10595-rk,gemat1_flops-vs-iter-4929x10595-cd}.

As discussed in \Cref{sec:imp_tricks}, the adaptive methods are typically more expensive than non-adaptive methods as one must update the sketched residuals $\mR_i^k$ for $i=1,\ldots,q$ at each iteration $k$. Yet even after taking flops into consideration, we find that the max-distance sampling strategy still performs the best overall. For randomized Kaczmarz applied to an overdetermined synthetic matrix, uniform sampling performance is comparable to max-distance (\Cref{synthetic_flops-vs-iter-1000x100-rk}). In all other experiments, however, max-distance sampling is the clear winner. Since max-distance sampling performs at least as well per iteration as capped adaptive sampling and sampling with probabilities proportional to the sketched losses, yet the max-distance sampling method is less expensive, it naturally performs the best among the adaptive methods when flop counts are considered.

\begin{figure}
    \begin{subfigure}{0.48\textwidth}
        \includegraphics[width=\textwidth]{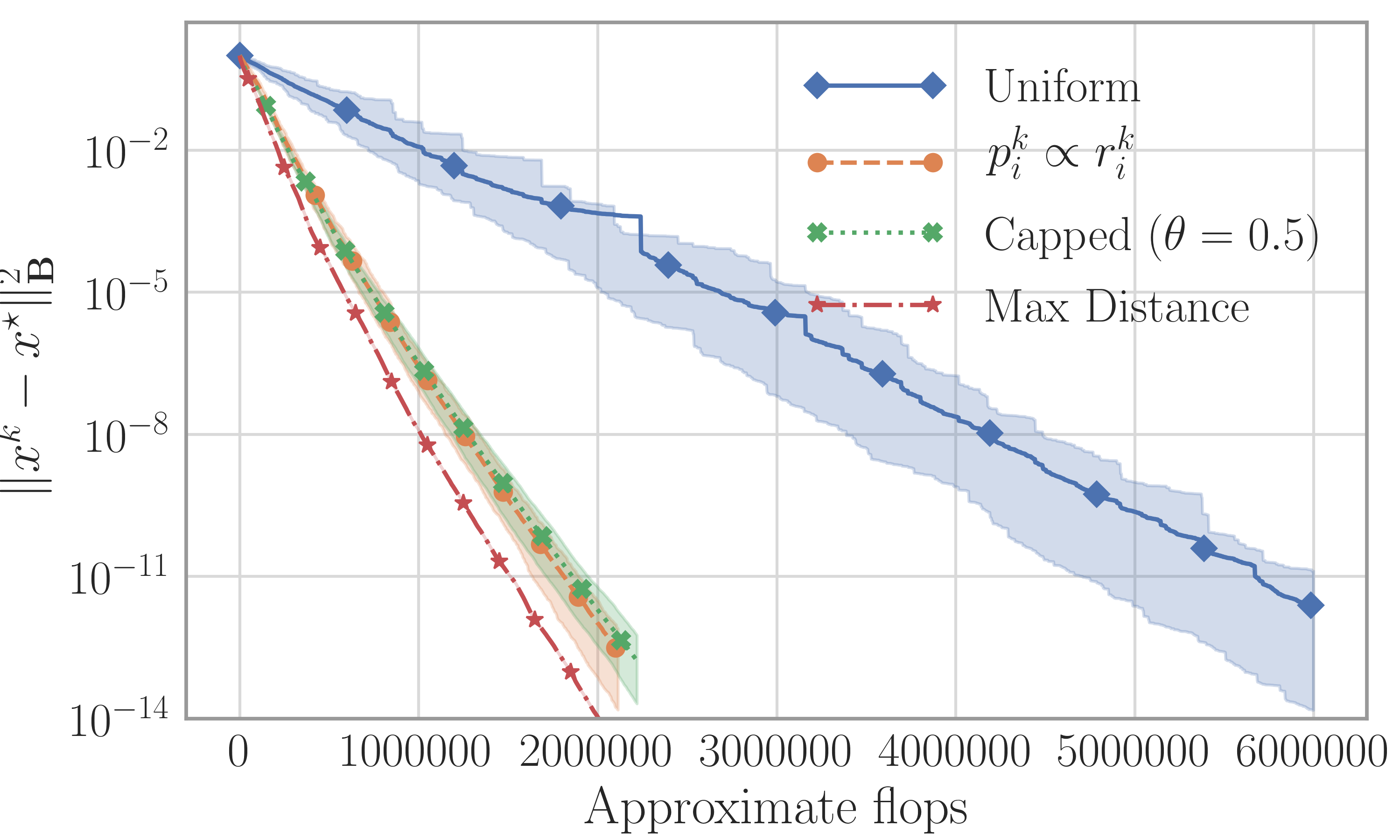}
        \caption{Adaptive randomized Kaczmarz, $\mA \in \R^{100 \times 1000}$.}
        \label{synthetic_flops-vs-iter-100x1000-rk}
    \end{subfigure}
    \hspace{0.01\textwidth} 
    \begin{subfigure}{0.48\textwidth}
        \includegraphics[width=\textwidth]{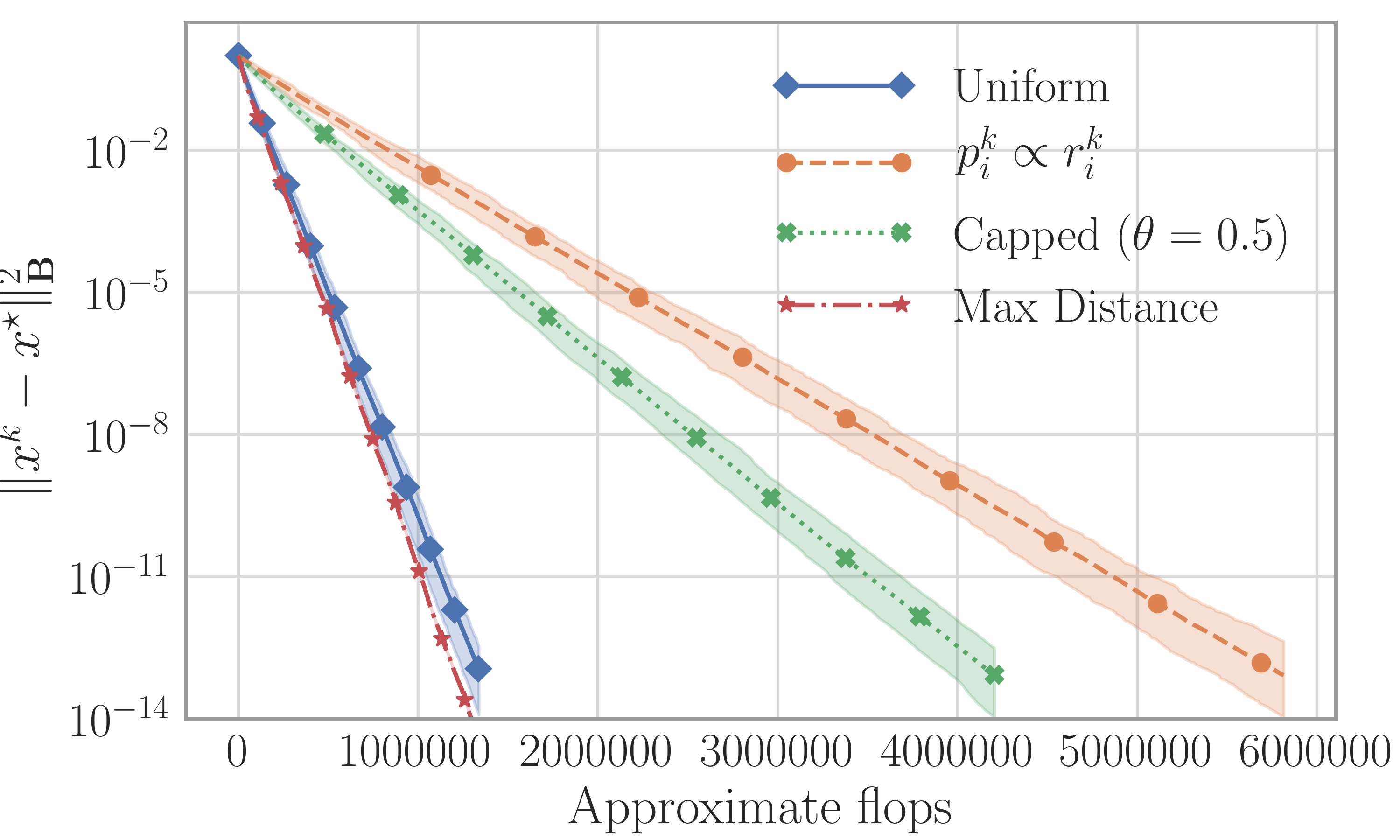}
        \caption{Adaptive randomized Kaczmarz, $\mA \in \R^{1000 \times 100}$.}
        \label{synthetic_flops-vs-iter-1000x100-rk}
    \end{subfigure}
    \begin{subfigure}{0.48\textwidth}
        \includegraphics[width=\textwidth]{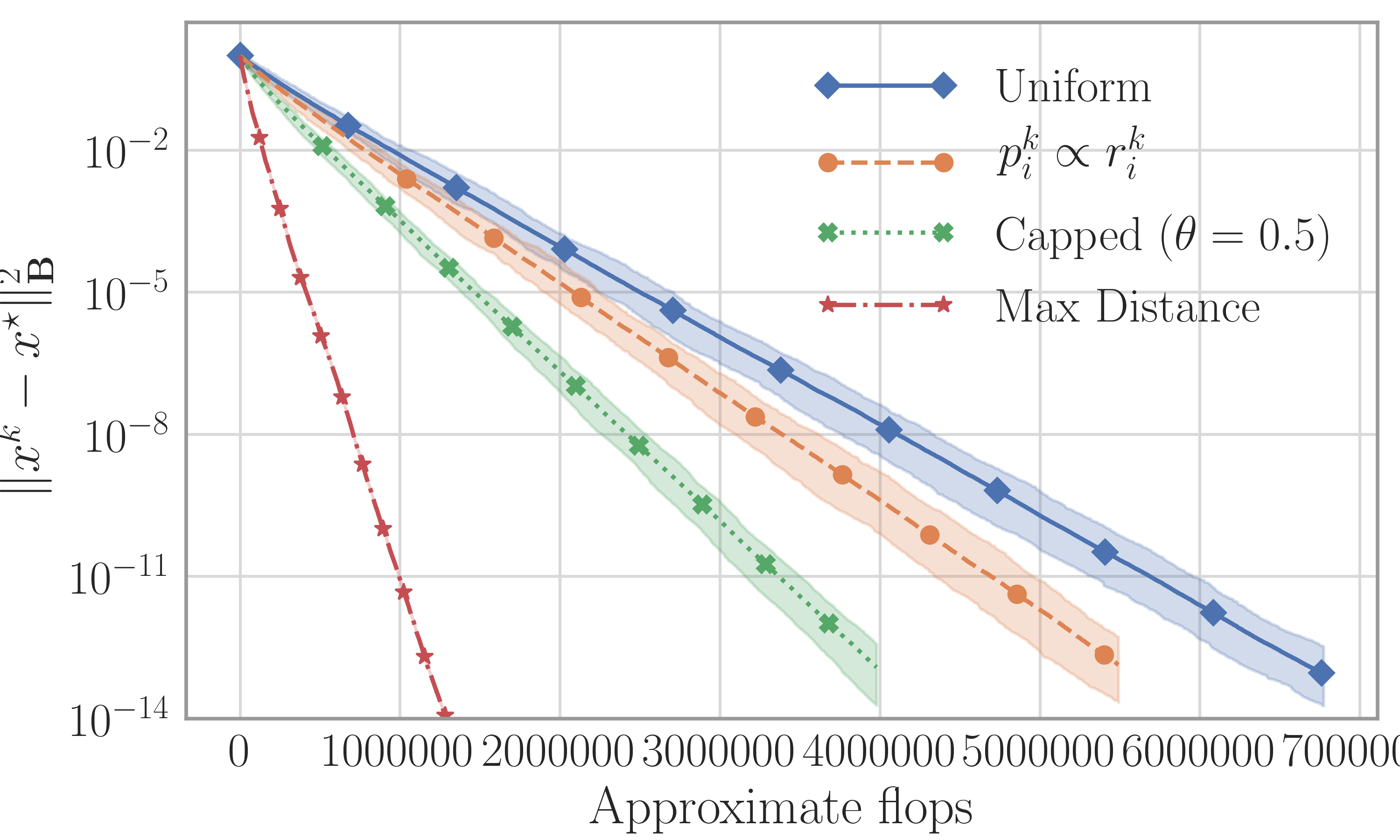}
        \caption{Adaptive coordinate descent, $\mA \in \R^{100 \times 1000}$.}
        \label{synthetic_flops-vs-iter-100x1000-cd}
    \end{subfigure}
    \hspace{0.01\textwidth}
    \begin{subfigure}{0.48\textwidth}
        \includegraphics[width=\textwidth]{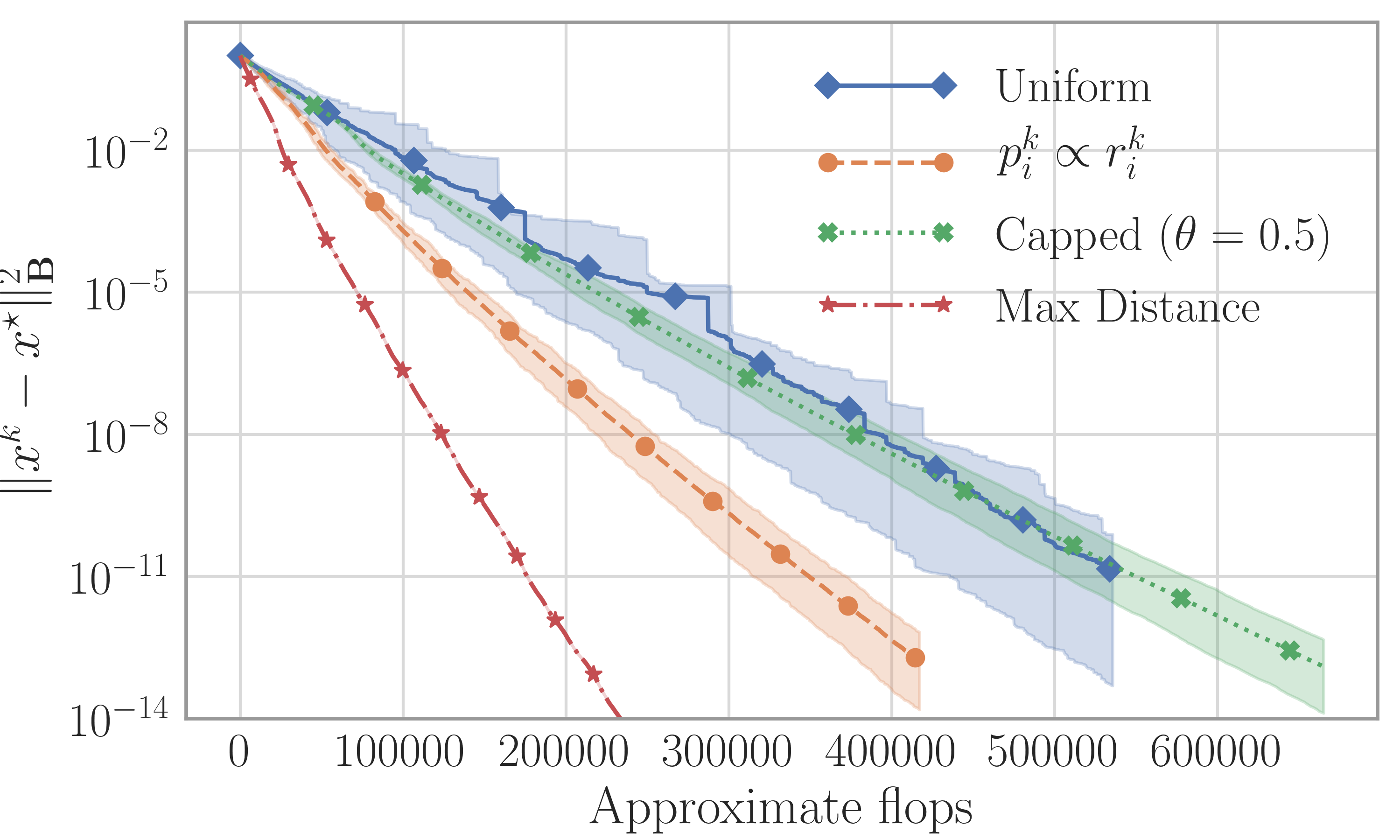}
        \caption{Adaptive coordinate descent, $\mA \in \R^{1000 \times 100}$.}
        \label{synthetic_flops-vs-iter-1000x100-cd}
    \end{subfigure}
\centering
\caption{A comparison between different selection strategies for randomized Kaczmarz and coordinate descent methods. Squared error norms were averaged over 50 trials and are plotted against the approximate flops aggregated over the computations that occur at each iteration. Confidence intervals indicate the middle 95\% performance. Subplots on the left show convergence for underdetermined systems, while those on the right show the convergence on an overdetermined systems.}
\label{fig:synth-err-vs-flops}
\end{figure}

\begin{figure}
    \begin{subfigure}{0.48\textwidth}
        \includegraphics[width=\textwidth]{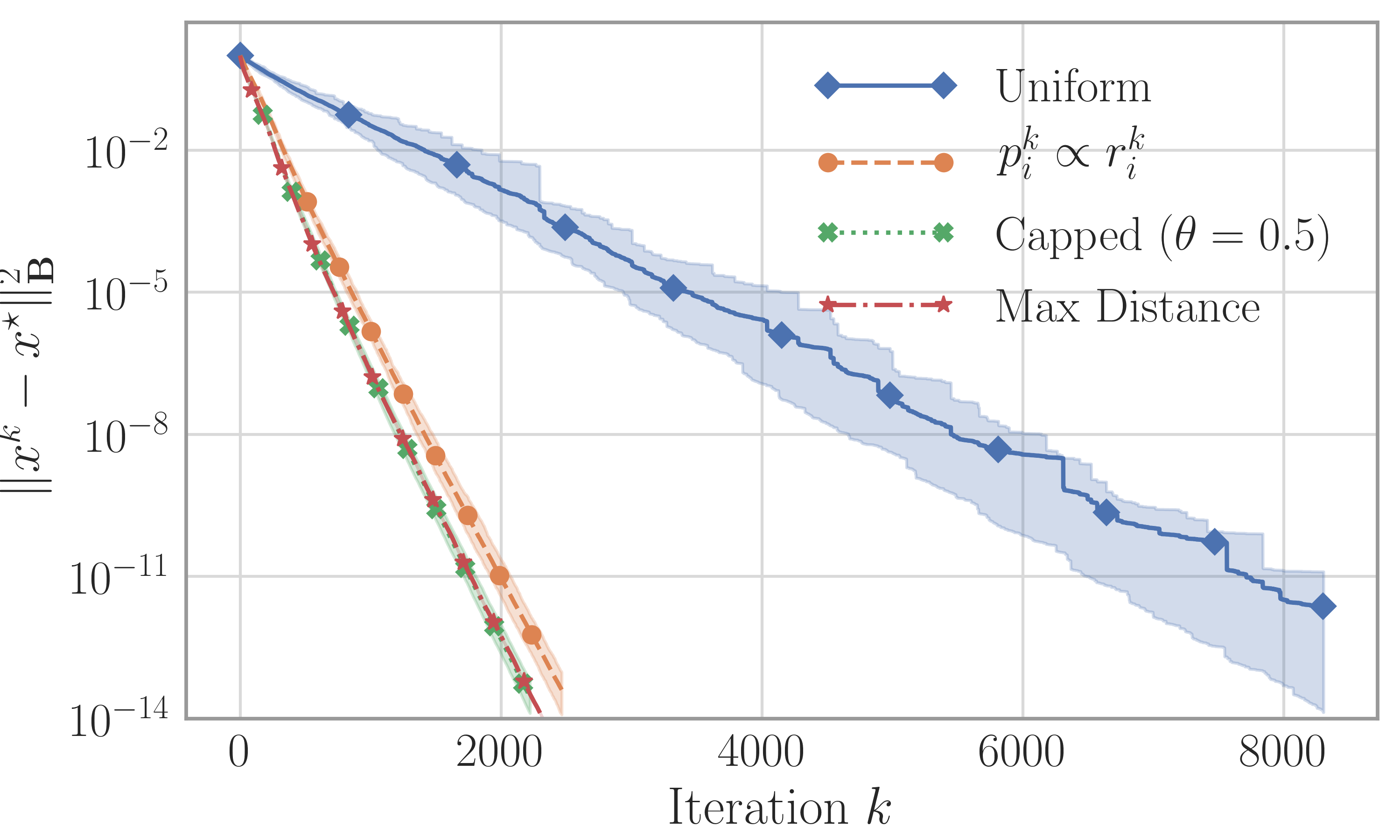}
        \caption{Adaptive coordinate descent.}
        \label{ash958_errsq-vs-iter-958x292-cd}
    \end{subfigure}
    \hspace{0.01\textwidth}
    \begin{subfigure}{0.48\textwidth}
        \includegraphics[width=\textwidth]{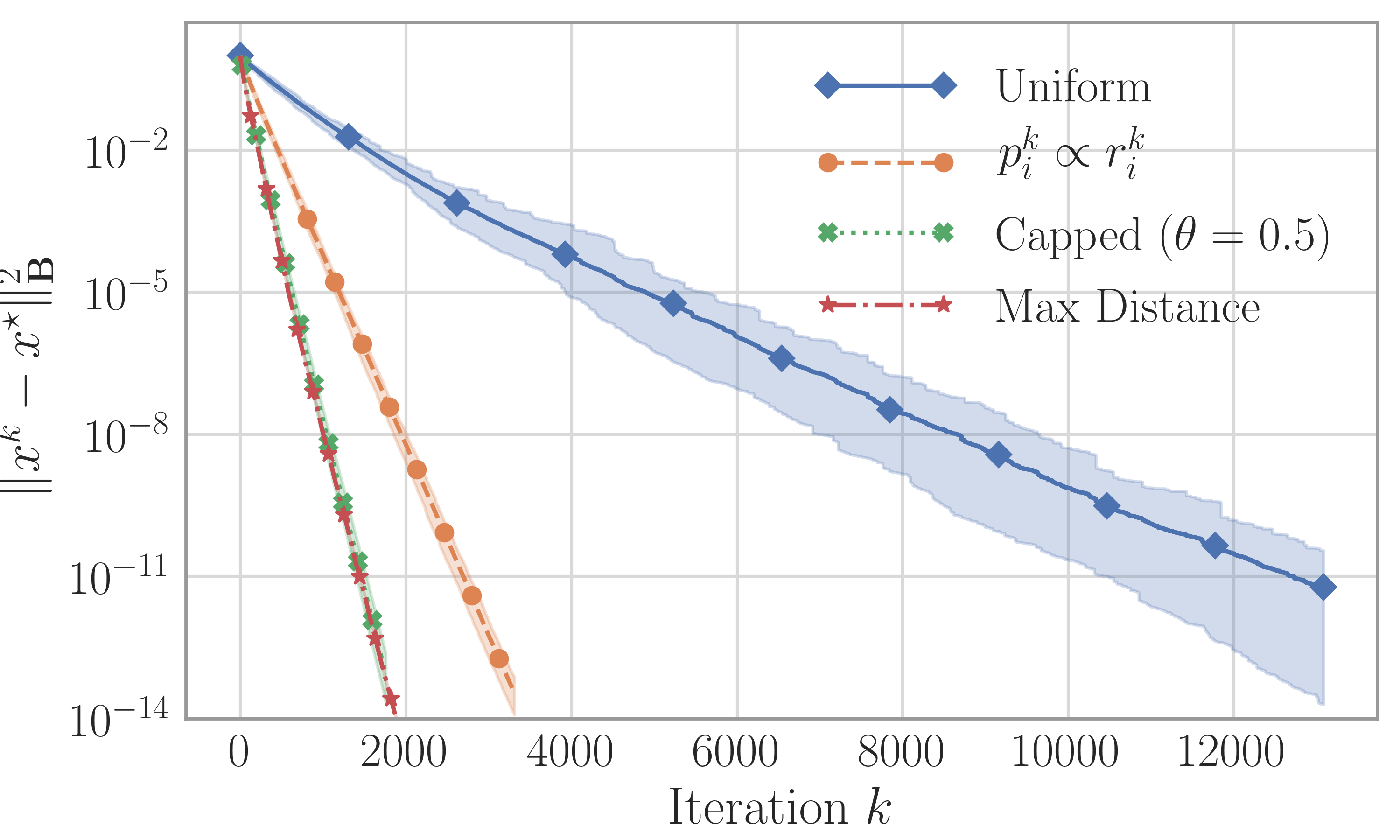}
        \caption{Adaptive randomized Kaczmarz.}
        \label{ash958_errsq-vs-iter-958x292-rk}
    \end{subfigure}
    \begin{subfigure}{0.48\textwidth}
        \includegraphics[width=\textwidth]{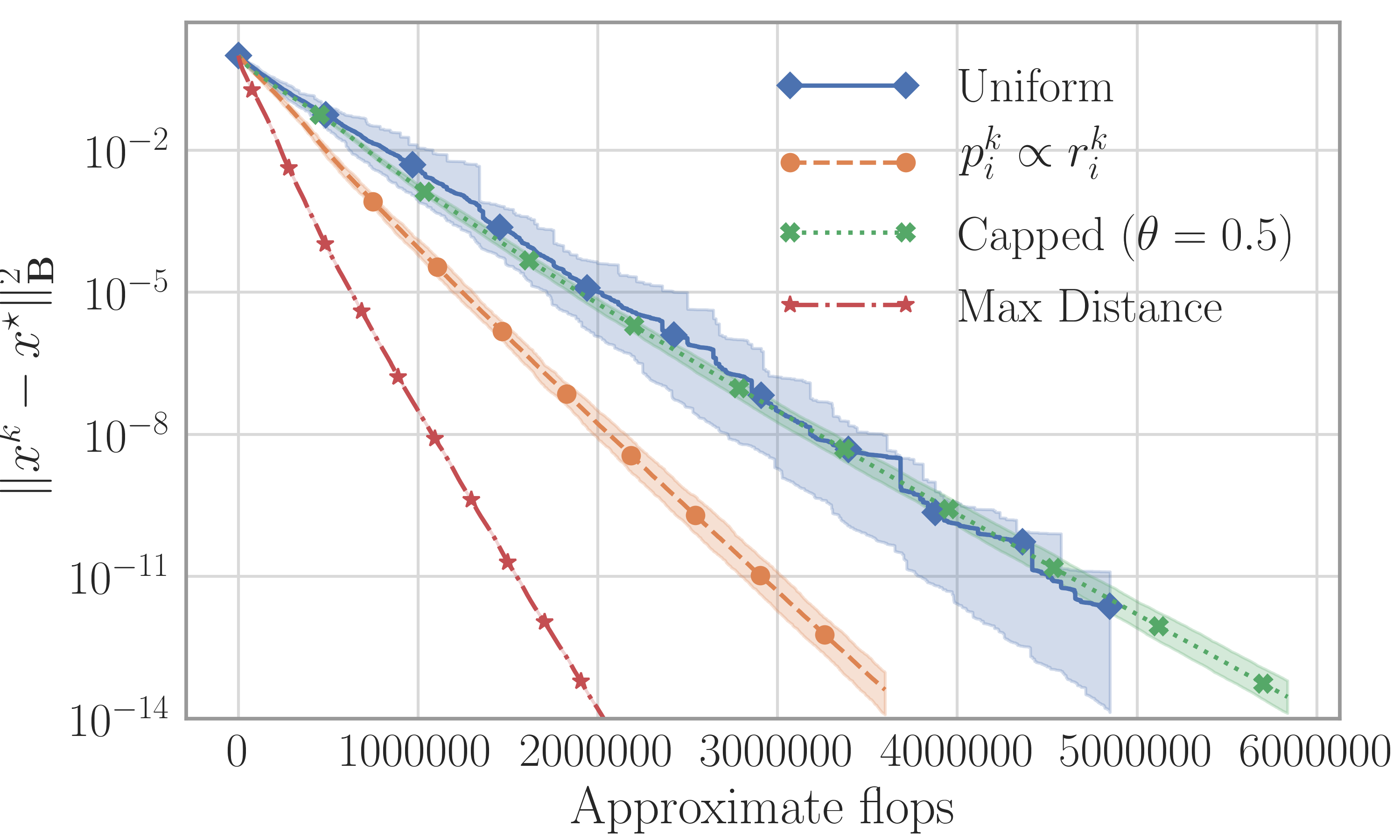}
        \caption{Adaptive coordinate descent.}
        \label{ash958_flops-vs-iter-958x292-cd}
    \end{subfigure}
    \hspace{0.01\textwidth}
    \begin{subfigure}{0.48\textwidth}
        \includegraphics[width=\textwidth]{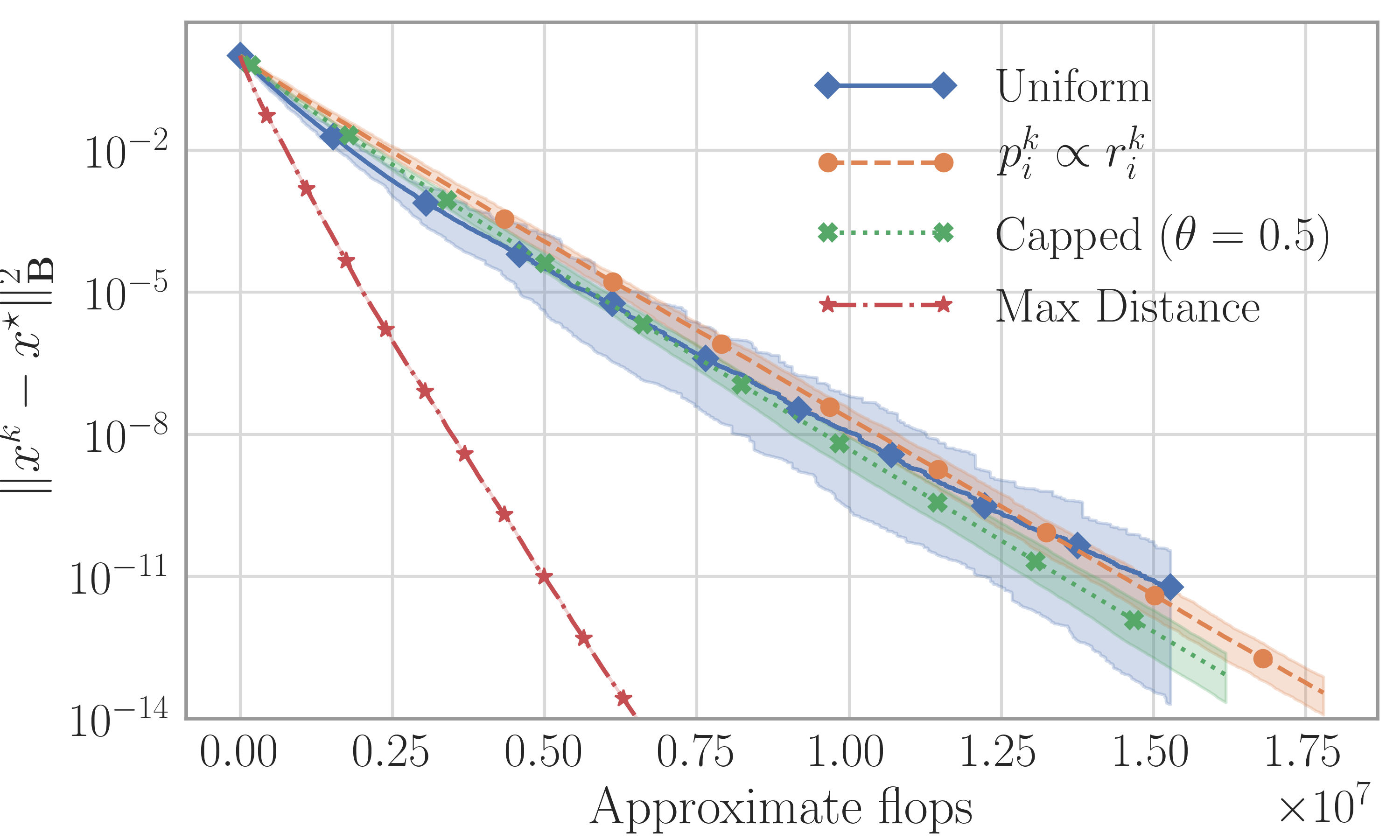}
        \caption{Adaptive randomized kaczmarz.}
        \label{ash958_flops-vs-iter-958x292-rk}
    \end{subfigure}
\centering
\caption{A comparison between different selection strategies for randomized Kaczmarz and coordinate descent methods on the Ash958 matrix. Squared error norms were averaged over 50 trials and plotted against both the iteration and the approximate flops required. Confidence intervals indicate the middle 95\% performance. }
\label{fig:ash_perf}
\end{figure}

\begin{figure}
    \begin{subfigure}{0.48\textwidth}
        \includegraphics[width=\textwidth]{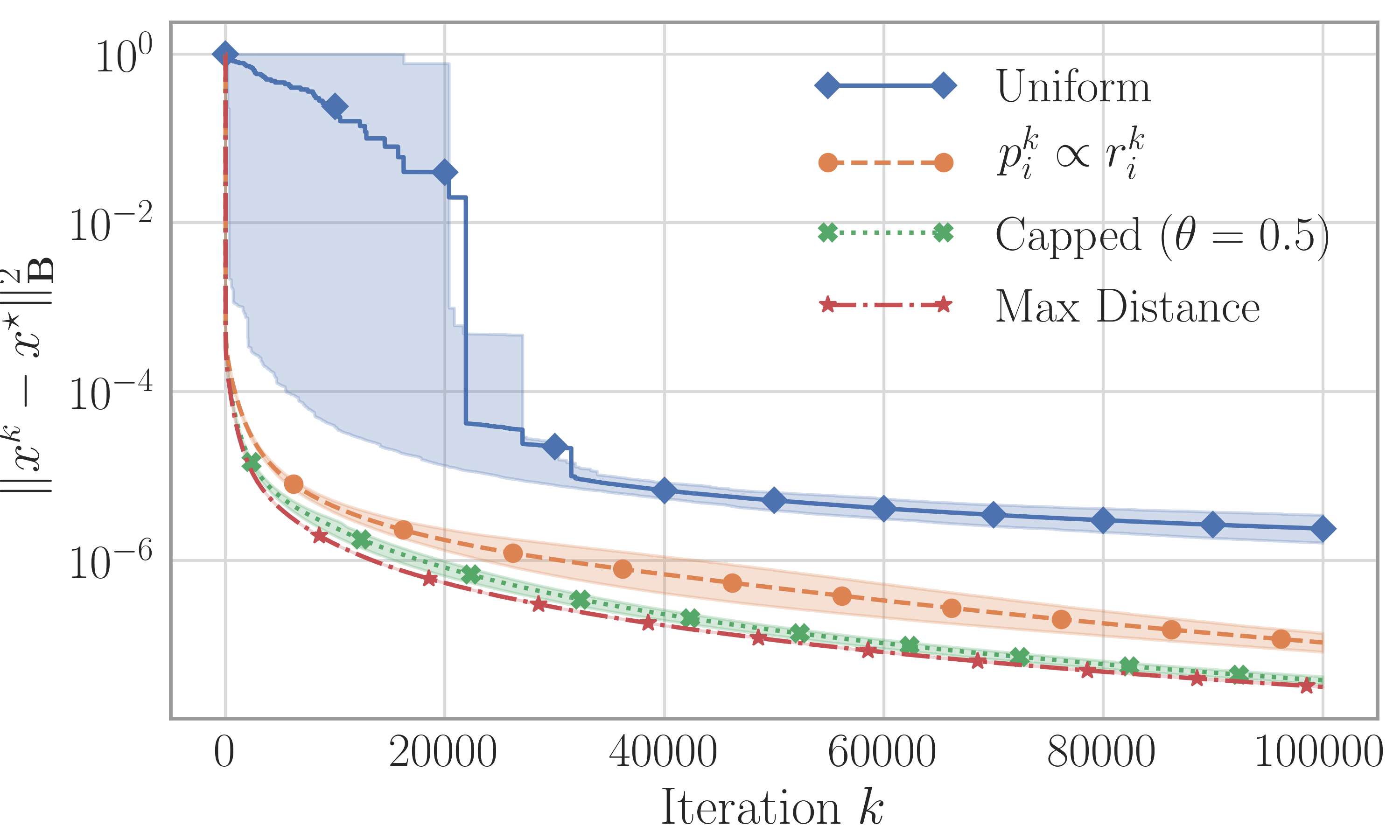}
        \caption{Adaptive randomized Kaczmarz.}
        \label{gemat1_errsq-vs-iter-4929x10595-rk}
    \end{subfigure}
    \begin{subfigure}{0.48\textwidth}
        \includegraphics[width=\textwidth]{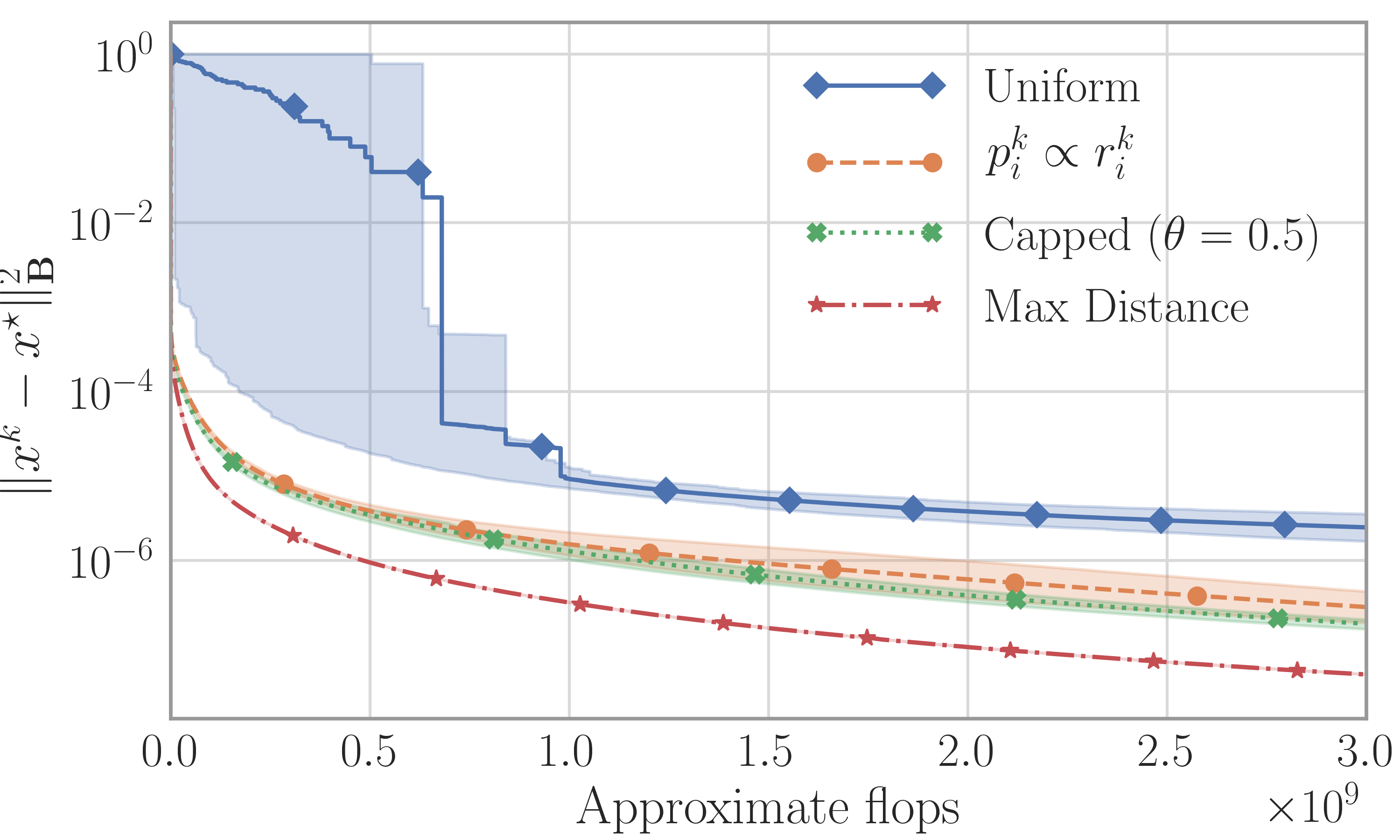}
        \caption{Adaptive randomized Kaczmarz.}
        \label{gemat1_flops-vs-iter-4929x10595-rk}
    \end{subfigure}
    \begin{subfigure}{0.48\textwidth}
        \includegraphics[width=\textwidth]{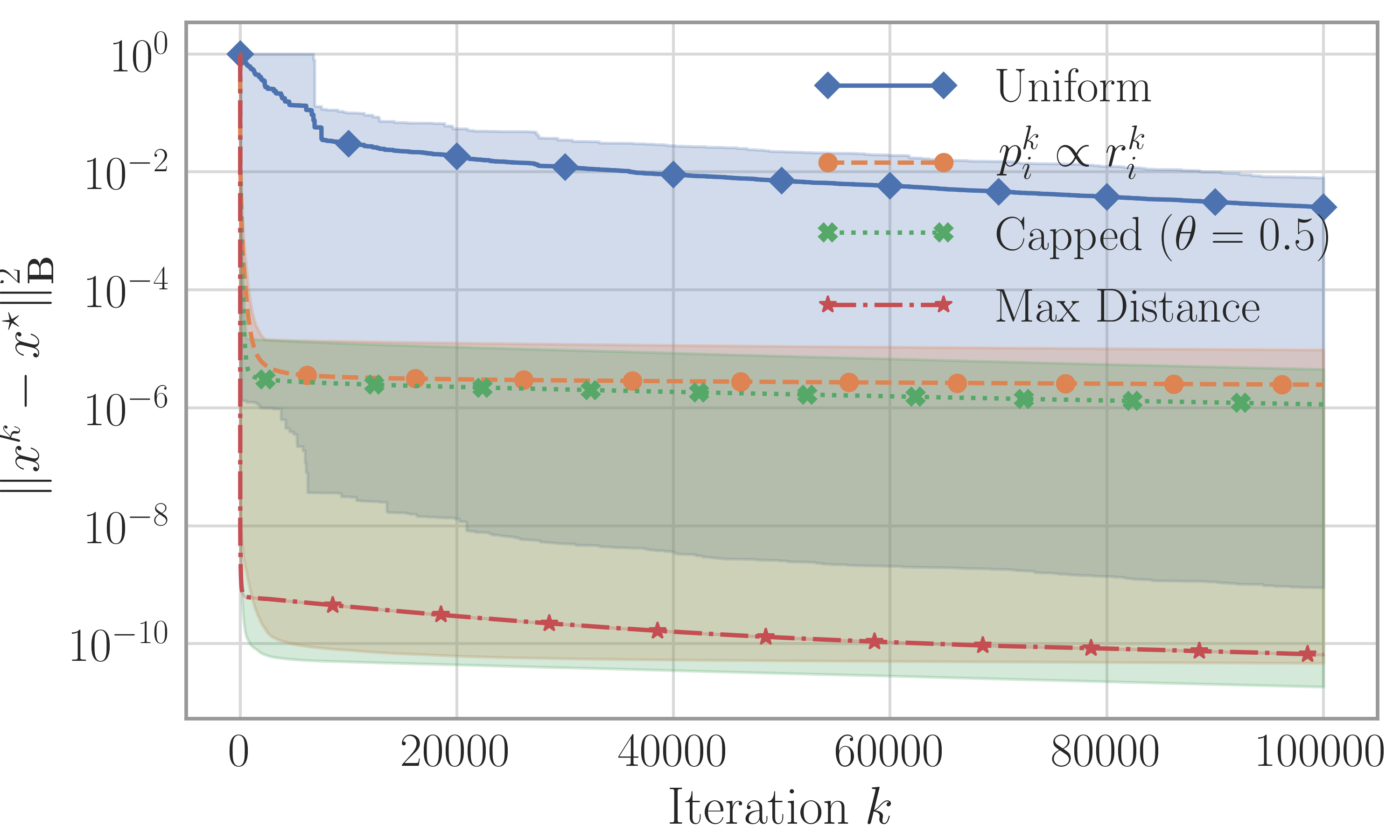}
        \caption{Adaptive coordinate descent.}
        \label{gemat1_errsq-vs-iter-4929x10595-cd}
    \end{subfigure}
    \begin{subfigure}{0.48\textwidth}
        \includegraphics[width=\textwidth]{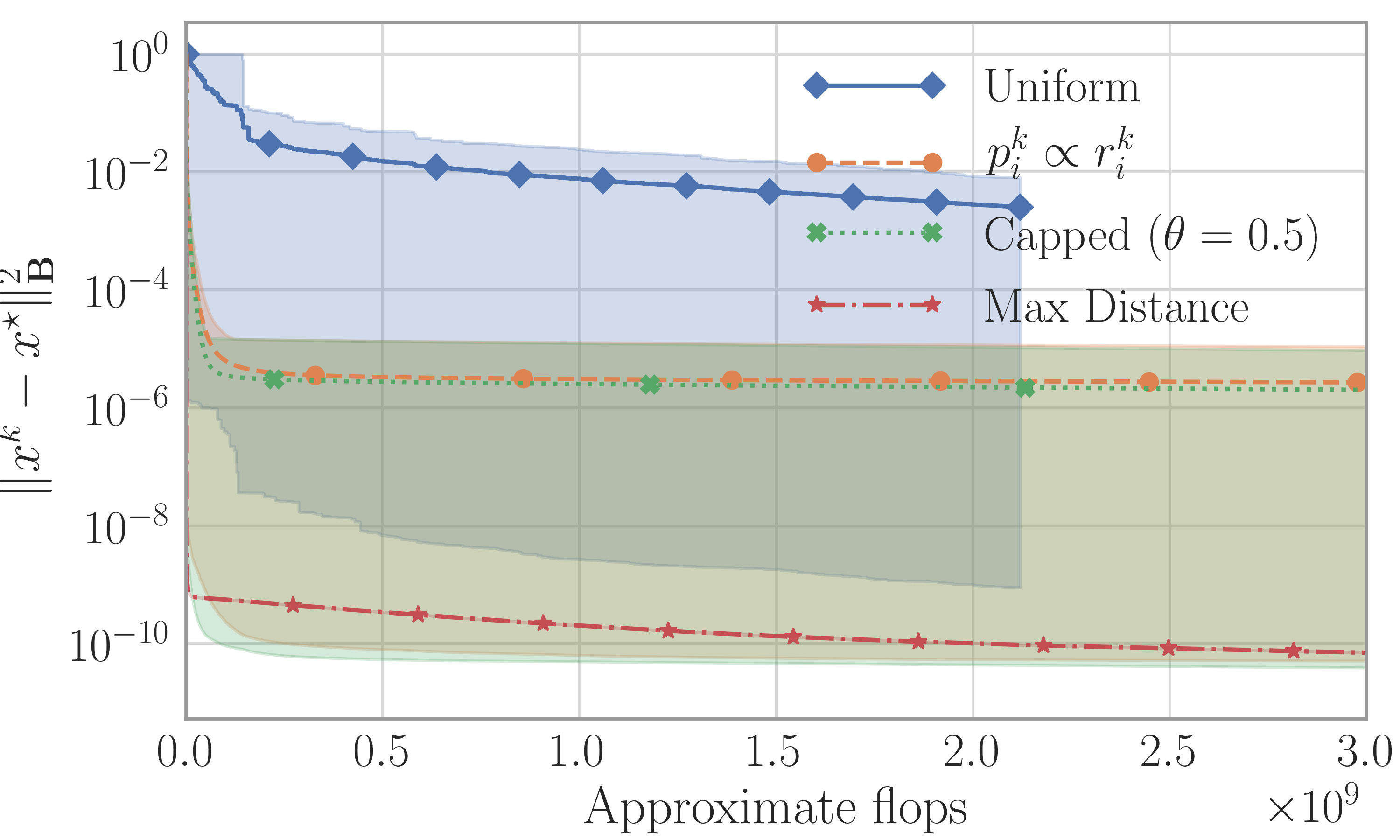}
        \caption{Adaptive coordinate descent.}
        \label{gemat1_flops-vs-iter-4929x10595-cd}
    \end{subfigure}
\centering
\caption{A comparison between different selection strategies for randomized Kaczmarz and coordinate descent on the GEMAT1 matrix. Squared error norms were averaged over 50 trials and plotted against both the iteration and the approximate flops required. Confidence intervals indicate the middle 95\% performance. }
\label{fig:gemat_perf}
\end{figure}

\subsection{Spectral constant estimates}

\Cref{thm:fixed,thm:prop_conv,thm:maxDistBetterThanUnif,thm:maxDistConv,thm:Bai,thm:RGRK} of \Cref{sec:convergence} provide conservative views of the convergence rates of each method, as the spectral constants of \Cref{defn:sigmas} give the expected convergence corresponding to the worst possible point $x\in \range{\mB^{-1}\mA}$ as opposed to the iterates $x^k$. In practice, the convergence at each iteration might perform better than the convergence bounds indicate.

Recall that the convergence rates derived in \Cref{sec:convergence} are given in terms of spectral constants (\Cref{defn:sigmas}) of the form 
\begin{equation*}
\sigma_{p}^2(\mB,\mS) \eqdef  
 \min_{x \in \range{\mB^{-1}\mA^\top}} \frac{\EE{i\sim p}{f_i(x)}}{\norm{x-x^\star}^2_{\mB}}.
\end{equation*}
We will refer to the value
\begin{equation*}
\frac{\EE{i\sim p^k}{f_i(x^k)}}{\norm{x^k-x^\star}^2_{\mB}}
\end{equation*}
as the \emph{expected step size factor} and note that larger values indicate superior performance. 

The smallest expected step size factor observed for each method provides an estimate and upper bound on the spectral constants in the derived convergence rates. 
The minimal expected step size factor for each sampling method applied to random Gaussian matrices of size $1000\times 100$ and $100\times 1000$ are reported in \Cref{tab:min_exp_step}. As expected, we find that these values increase from uniform sampling, sampling proportional to the sketched losses, capped adaptive sampling and finally max-distance selection. In \Cref{thm:prop_conv}, we proved a bound on the convergence rate for sampling proportional to the sketched losses that was twice as fast as the convergence guarantee for uniform sampling. We find that the estimated spectral constants in \Cref{tab:min_exp_step} for the proportional sampling strategy is also at least twice as large as the estimated spectral constant for uniform sampling.

\begin{table}[]
    \centering
    \begin{tabular}{|c|c|c|c|c|}
    \hline
        \multirow{2}{*}{\textbf{Sampling}} &
         \multicolumn{2}{|c|}{\textbf{Randomized Kaczmarz}}  &  \multicolumn{2}{|c|}{\textbf{Coordinate Descent}}\\
         \cline{2-5}
         & $1000\times 100$ & $100\times 1000$ & $1000\times 100$ & $100\times 1000$
        \\        \hline \hline
        Uniform & 0.00705 & 0.00667 & 0.00656 &  0.00715 \\ \hline
        $p_i\propto \norm{\mA_{:i}}_2^2$ &  0.02019 & 0.01569 &  0.01722 &  0.02014 \\ \hline
        Capped & 0.03885 & 0.01901 & 0.01952 & 0.03878
        \\ \hline
        Max-distance & 0.04593 & 0.01994 & 0.02171 &  0.04711\\ \hline
    \end{tabular}
    \caption{Minimal expected step size factor for each sampling method applied to matrices containing i.i.d.\ guassian entries.}
    \label{tab:min_exp_step}
\end{table}

\section{Conclusions}

We extend adaptive sampling methods to the general sketch-and-project setting. We present a computationally efficient method for implementing the adaptive sampling strategies using an auxiliary update. For several specific adaptive sampling strategies including max-distance selection, the capped adaptive sampling of \cite{BaiWuSISC2018,BAIWu201821}, and sampling proportional to the sketched residuals, we derive convergence rates and show that the greedy max-distance sampling rule has the fastest convergence guarantee among the sampling methods considered. This superior performance is seen in practice as well for both the randomized Kaczmarz and coordinate descent subcases.

\appendix

\section{Implementation tricks and computational complexity}\label{sec:imp_tricks_full}

We describe how one can perform adaptive sketching with the same order of cost per iteration as the standard non-adaptive sketch-and-project method when $\tau q$, the number of sketches $q$ times the sketch size $\tau$, is not significantly larger than the number of columns $n$. In particular, we show how adaptive sketching methods can be performed for a per-iteration cost of $O(\tau^2 q + \tau n)$, whereas the standard non-adaptive sketch-and-project method has a per-iteration cost of $O( \tau n)$. 
The precomputations and efficient update strategies presented here are a generalization of those suggested in \cite{BaiWuSISC2018} for the Kaczmarz setting. The computational costs given in this section may be over-estimates of the costs required for specific sketch choices such as when the update is sparse, as is the case in coordinate descent. The special cases of adaptive Kaczmarz and adaptive coordinate descent are analyzed in \Cref{sec:costs-and-convergence-summaries}.

Pseudocode for efficient implementation is provided in \Cref{alg:adasketch}. Throughout this section, we will frequently omit $O(1)$ and $O(\log(q))$ flop counts since they are insignificant compared to the number of rows $m$, the number of columns $n$, and the number of sketches $q$.

\subsection{Per-iteration cost}\label{subsec:per_iter_cost}
The main computational costs of adaptive sketch-and-project (\Cref{alg:adaSKep}) at each iteration come from computing the sketched losses $f_i(x^k)$ of \Cref{eqn:residualk} and updating the iterate from $x^k$ to $x^{k+1}$ via \Cref{eqn:xupdate}. We now discuss how these steps can be calculated efficiently. A suggested efficient implementation for adaptive sketch-and-project is provided in \Cref{alg:adasketch}. The costs of each step of an iteration of the adaptive sketch-and-project method are summarized in \Cref{tab:general-shared-costs}.

Let $\mC_i$ be any square matrix satisfying
\begin{equation}\label{def:C_i}
    \mC_i \mC_i^\top = (\mS_{i}^\top \mA \mB^{-1} \mA^\top \mS_{i})^\dagger.
\end{equation}
For example, $\mC_i$ could be the Cholesky decomposition of $(\mS_{i}^\top \mA \mB^{-1} \mA^\top \mS_{i})^\dagger$. The sketched loss $f_i(x^k)$ and the iterate update from $x^k$ to $x^{k+1}$ can now be written as
\[f_i(x^{k}) = \norm{\mS_i^\top( \mA x^k - b)}^2_{\mC_i \mC_i^\top} = \norm{\mC_i^\top \mS_i^\top( \mA x^k - b)}^2_2\]
and
\[x^{k+1} = x^k - \mB^{-1}\mA^\top \mS_{i_k} \mC_{i_k} \mC_{i_k}^\top \mS_{i_k}^\top(\mA x^k-b).\]
Notice that both the iterate update for $x^k$ and the formula for the sketched loss $f_i(x^k)$ share the sketched residual  $\mR_i^k\eqdef  \mC_i^\top \mS_i^\top( \mA x^k - b)$ defined in \Cref{def:sketched-res}. 
In adaptive methods one must compute the sketched residual $\mR_i^k$ for $i=1, 2, \ldots, q$. When sampling from a fixed distribution, however, calculating the sketched losses $f_i(x^k)$ is unnecessary and only the sketched residual $\mR_{i_k}^k$ corresponding to the selected index $i_k$ need be computed.

Depending on the sketching matrices $\mS_i$ and the matrix $\mB$, it is possible to update the iterate $x^k$ and compute the sketched losses $f_i(x^k)$ more efficiently if one maintains the set of sketched residuals
$\{\mR_i^k : i = 1, 2, \ldots, q\}$ in memory.
 Using the sketched residuals, the calculations above can be rewritten as 
\begin{equation}
    f_i(x^{k}) = \norm{\mR_i^k}^2_2 \label{eqn:sketched-loss-efficient}
\end{equation}
and 
\begin{equation}
    x^{k+1} = x^k - \mB^{-1}\mA^\top \mS_{i_k}\mC_{i_k} \mR_{i_k}^k \label{eqn:iterate-efficient}.
\end{equation}

The sketched residuals $\{\mR_i^k : i = 1, 2, \ldots, q\}$ can either be computed via an auxiliary update applied to the set of previous set of sketched residuals $\{\mR_i^{k-1} : i = 1, 2, \ldots, q\}$ or directly using the iterate $x^k$. Using the auxiliary update,
\begin{align}
    \mR_i^{k+1} &= \mC_i^\top\mS_i^\top( \mA x^{k+1} - b) \nonumber \\
    &= \mC_i^\top\mS_i^\top\Big( \mA (x^k - \mB^{-1}\mA^\top \mS_{i_k}\mC_{i_k} \mR_{i_k}^k) - b\Big) \nonumber \\
    &= \mR_i^{k} - \mC_i^\top\mS_i^\top \mA \mB^{-1}\mA^\top \mS_{i_k}\mC_{i_k} \mR_{i_k}^k\label{eqn:residual-efficient}
\end{align}
with the initialization
\[
\mR_i^0 = \mC_i^\top\left(\mS_i^\top( \mA x^0 - b)\right).
\]
If the matrix $\mC_i^\top\mS_i^\top \mA \mB^{-1}\mA^\top \mS_{j}\mC_{j} \in \R^{\tau \times \tau}$ is precomputed for each $i, j = 1, 2, \ldots, q$, the sketched residual $\mR_i^k$ can be updated to $\mR_i^{k+1}$ for $2 \tau^2 $ flops for each index $i$ via \Cref{eqn:residual-efficient}. Using the precomputed matrices requires storing $\tfrac{1}{4}\tau(\tau + 1)q(q+1)$ floats. 
 
In the non-adaptive case, one only needs to compute the single sketched residual $\mR_{i_k}^k$ as opposed to the entire set of sketched residuals, since the sketched losses $f_i(x^k)$ are not needed.
If the matrices
\begin{equation*}
\mC_{i}^\top \mS_{i}^\top \mA \in \R^{\tau \times n} \quad \text{and}\quad \quad \mC_{i}^\top \mS_{i}^\top b \in \R^{\tau},
\end{equation*}
are precomputed for $i=1,2,\ldots,q$, computing each sketched residual $\mR_{i}^k$ directly from the iterate $x^k$ costs $2 \tau n$ flops via \Cref{def:sketched-res}. If $q \tau > n$, then it is cheaper to compute the sketched residual $\mR_{i_k}^k$ using the auxiliary update~\Cref{eqn:residual-efficient} rather than computing it directly from $x^k$.

From the sketched residual $\mR_i^k$, the sketched losses $f_i(x^k)$ can be computed for $2 \tau - 1$ flops for each index $i$ via \Cref{eqn:sketched-loss-efficient}. If the matrix $\mB^{-1}\mA^\top \mS_{i}\mC_{i} \in \R^{n \times \tau}$ is precomputed for each $i = 1, 2, \ldots, q$, the iterate $x^k$ can then be updated to $x^{k+1}$ for $2 \tau n $ flops via \Cref{eqn:iterate-efficient}. These costs are summarized in \Cref{tab:general-shared-costs}.

\begin{algorithm}[!t]
\begin{algorithmic}[1]
\State \textbf{input:} $\mA \in \R^{m\times n}$, $b\in \R^{m}$, $\{\mS_i \in \R^{m \times \tau } : i = 1, 2, \ldots, q\}$, $\mB \in \mR^{n \times n}$, $x^0\in \range{\mB^{-1}\mA^\top},$ 
\State compute $\mC_i = \operatorname{Cholesky}\Big((\mS_{i}^\top \mA \mB^{-1} \mA^\top \mS_{i})^\dagger\Big)$ for $i = 1, 2, \ldots, q$
\Statex \Comment{The $\mC_i$ can be discarded after Line~\ref{ln:efficient-end-init}.}
\State compute $\mB^{-1}\mA^\top \mS_{i}\mC_{i} \in \R^{n \times \tau}$ for $i = 1, 2, \ldots, q$
\State compute $\mC_i^\top\mS_i^\top \mA \mB^{-1}\mA^\top \mS_{j}\mC_{j} \in \R^{\tau \times \tau}$ for $i,j = 1, 2, \ldots, q$
\State initialize $\mR_i^0 = \mC_i^\top\left(\mS_i^\top( \mA x^0 - b)\right) \in \R^{\tau} $ for $i = 1, 2, \ldots, q$ \label{ln:efficient-end-init} 
\For {$k = 0, 1, 2, \dots$} 
	\State compute $f_i(x^k) = \norm{\mR_i^{k}}_2^2$ for $i = 1, 2, \ldots, q$ \label{ln:efficient-sketched-loss} 
	\State sample $i_k \sim p_i^k$, where $p^k\in \Delta_q$ is a function of $f(x^k)$ \label{ln:efficient-sampling} 
	\State update $x^{k+1} = x^k - (\mB^{-1}\mA^\top \mS_{i_k}\mC_{i_k})  \mR_{i_k}^k$ \label{ln:efficient-iterate-update} 
    \State update $\mR_i^{k+1} =\mR_i^{k} - (\mC_i^\top\mS_i^\top \mA \mB^{-1}\mA^\top \mS_{i_k}\mC_{i_k}) \mR^{k}_{i_k}$ for $i = 1, 2, \ldots, q$ \label{ln:efficient-auxiliary-update} 
\EndFor
\State \textbf{output:} last iterate $x^{k+1}$
\end{algorithmic}
\caption{Efficient Adaptive Sampling Sketch-and-Project}
\label{alg:adasketch}
\end{algorithm}

\begin{table}[h]
\centering
    \begin{subtable}[t]{.38\linewidth}
        \centering
        \begin{tabular}{|c|c|}
            \hline
            \textbf{\makecell{Per iteration\\computation}} & \textbf{Flops} \\ \hline \hline
            \makecell{$f_i(x^k) \; \forall i$ via \\ \Cref{eqn:sketched-loss-efficient}} & $ (2 \tau - 1) q $ \\ \hline 
            \makecell{$x^{k+1}$ via \\ \Cref{eqn:iterate-efficient}} & $2 \tau n$ \\ \hline 
            \makecell{$\mR_i^k \ \forall i$ with\\auxiliary update,\\
            \Cref{eqn:residual-efficient}} & $2 \tau^2 q $ \\ \hline
            \makecell{$\mR_{i_k}^k$ via direct \\ computation,\\
            \Cref{def:sketched-res}} & $2 \tau n$\\ \hline
        \end{tabular}
        \caption{Baseline flop counts. Flop counts of $O(1)$ have been omitted.}
        \label{tab:general-shared-flop-costs}
    \end{subtable} 
    \begin{subtable}[t]{.6\linewidth}
        \centering
        \begin{tabular}{|c|c|}
            \hline
            \textbf{Stored Object} & \textbf{Storage} \\ \hline \hline
            $x^k$  & $n$ \\ \hline
            $\mR_i^k \quad \forall i$  & $\tau q $ \\ \hline
            $\mB^{-1}\mA^\top \mS_{i}\mC_{i}\quad \forall i$  & $ \tau q n$ \\ \hline
            \makecell{$\mC_i^\top\mS_i^\top \mA \mB^{-1}\mA^\top \mS_{j}\mC_{j}$ \\ $\forall i,j$} & $\tfrac{1}{4}\tau(\tau + 1)q(q+1)$ \\ \hline
            \makecell{$\mC_{i}^\top \mS_{i}^\top \mA \ $  and  \\  $\mC_{i}^\top \mS_{i}^\top b \quad \forall i$} & $\tau q (n + 1)$\\ \hline
        \end{tabular}
        \caption{Storage costs.}
    \end{subtable}%
    \caption{Summary of the costs of the of \Cref{alg:adasketch} excluding costs that are specific to the sampling method. The number of sketches is $q$, the sketch size is $\tau$ and the number of columns in the matrix $\mA$ is $n$.}
    \label{tab:general-shared-costs}
\end{table}

\subsection{Cost of sampling indices}

The cost of computing the sampling probabilities $p^k$ from the sketched losses $f_i(x^k)$ depends on the sampling strategy used. Sampling from a fixed distribution can be achieved with an $O(1)$ cost using precomputations of $O(q)$ \cite{Walker74}. 
Adaptive strategies sample from a new, unseen distribution at each iteration, which can be achieved with an average of $q$ flops using, for example, inversion by sequential search \cite[p.~86]{kemp1981efficient,devroye1986non-uniform}.  
In practice, the probabilities $p_i^k$ corresponding to each index $i$ are given by a function of the sketched losses $f(x_i^k)$ and normalizing these values is unnecessary. Instead, one can sum the $q$ sketched losses and apply inversion by sequential search with a random value $r$ generated between zero and the sum of these values. This summation requires $q-1$ flops. 
Thus the total cost for sampling from an adaptive probability distribution for the methods considered to approximately $2q$ flops on average. The costs for the sampling strategies discussed in \Cref{sec:methods} are summarized in \Cref{tab:general-rule-costs}. The calculations of these costs are discussed in more detail in \Cref{sec:sample_spec_costs}.

\begin{table}[h]
\centering
    \begin{tabular}{|c|c|c|}
        \hline
        \textbf{Sampling Strategy} & \textbf{Non-Sampling Flops} & \textbf{\makecell{Flops from Sampling}}\\ \hline \hline
        Fixed, $p_i^k \equiv p_i \ \forall k$ & $2 \tau \min(n, \tau q) + 2 \tau n$ & $O(1)$\\ \hline 
        Max-distance & \multirow{3}{*}{$(2 \tau^2 + 2 \tau - 1) q + 2 \tau n$} & \makecell{$q$ if $\tau > 1$ \\ $O(\log(q))$ if $\tau = 1$}\\ \cline{1-1}\cline{3-3}
        $p_i^k \propto f_i(x^k)$ &  & $2q$\\ \cline{1-1}\cline{3-3}
        Capped &  & $6q$\\  \hline
    \end{tabular}
    \caption{Rule-specific per-iteration costs of \Cref{alg:adasketch}. Only leading order flop counts are reported. The non-sampling flops are those that are independent of the specific adaptive sampling method used and are those that correspond to the steps indicated in \Cref{tab:general-shared-flop-costs}. The extra flops for sampling are those that are required to calculate the adaptive sampling probabilities $p^k$ at each iteration. The number of sketches is $q$, the sketch size is $\tau$ and the number of columns in the matrix $\mA$ is $n$.}
    \label{tab:general-rule-costs}
\end{table}

\begin{table}[]
    \centering
    \begin{tabular}{|c|c|c|c|c|}
    \hline
        \textbf{\makecell{Sampling\\Strategy}} 
& \textbf{\makecell{Flops Per \\ Iteration \\ When $\tau > 1$}} & \textbf{\makecell{Flops Per \\ Iteration \\ When $\tau = 1$}}
        \\        \hline \hline
        Fixed, $p_i^k \equiv p_i$ 
& $2 \tau \min(n, \tau q) + 2 \tau n$ & $2 \min(n, q) + 2 n$\\ \hline
        Max-distance 
& $(2 \tau^2 + 2 \tau)q + 2 \tau n$ & $3 q + 2 n$\\ \hline
        $p_i^k\propto f_i(x^k)$ 
& $(2 \tau^2 + 2 \tau+ 1)q + 2 \tau n$ & $5q + 2 n$\\ \hline
        Capped 
& $(2 \tau^2 + 2 \tau+ 5)q + 2 \tau n$ & $9q + 2 n$\\ \hline
    \end{tabular}
    \caption{Summary of convergence guarantees of \Cref{sec:convergence}, where $\gamma= 1/ \underset{i=1,\ldots, m}{\max}\sum_{j=1, j\ne i}^m p_i$ as defined in \Cref{eqn:gamma1} and $\epsilon = \theta(\gamma-1) \le \theta \tfrac{1}{m}$. Flop counts of $O(\log(q))$ have been omitted. Flop counts assume all matrices are dense. The number of sketches is $q$, the sketch size is $\tau$ and the number of columns in the matrix $\mA$ is $n$.}
    \label{tab:flop_summary}
\end{table}

\section{Sampling strategy specific costs}\label{sec:sample_spec_costs}

We detail the calculations that lead to the costs associated with each of the specific sampling strategies that are reported in \Cref{tab:general-rule-costs}.

\subsubsection{Sampling from a fixed distribution}
When sampling the indices $i$ from a fixed distribution, computing the sketched losses $f_i(x^k)$ is unnecessary and only the sketched residual $\mR_{i_k}^k$ of the selected index $i_k$ is needed to update the iterate $x^k$. 
If $q \tau > n$, where $q$ is the number of sketches, $\tau$ is the sketch size and $n$ is the number of columns in the matrix $\mA$, it is cheaper to compute the sketched residual $\mR_{i_k}^k$ using the auxiliary update~\Cref{eqn:residual-efficient} rather than computing it directly from $x^k$. 
Ignoring the $O(1)$ cost of sampling from the fixed distribution, the iterate update takes either $4 \tau n$ flops if $q \tau > n$ and one maintains the set of sketched residuals via the auxiliary update \Cref{eqn:residual-efficient} or $2 \tau (n+q)$ flops if the sketched residual $\mR_{i_k}^k$ is calculated from the iterate $x^k$ directly via \Cref{def:sketched-res}.

\subsubsection{Max-distance selection}
Performing max-distance selection requires finding the maximum element of the length $q$ vector of sketched losses given in \Cref{eqn:sketched-loss-efficient}. In the average case, this costs $q + O(\log q)$ flops, where $q$ flops are used to check each element and $O(\log q)$ flops arise from updates to the running maximal value. For convenience, we ignore the $O(\log q)$ flops and consider the cost of the selection step using the max-distance rule to be $q$ flops. 
If the sketches $\mS_i$ are vectors, or equivalently we have $\tau=1$, then the sketched residuals $\mR_i^k$ are scalars and finding the maximal sketched loss $f_i(x^k)$ is equivalent to finding the sketched residual $\mR_i^k$ of maximal magnitude. We can thus save $q$ flops per iteration by skipping the step of computing the sketched losses and instead taking the sketched residual of maximal magnitude.

\subsubsection{Sampling proportional to the sketched loss}\label{sec:propto-special-cost}

Sampling indices with probabilities proportional to the sketched losses $f_i(x^k)$ requires approximately $2q$ flops on average using inversion by sequential search. 

\subsubsection{Capped adaptive sampling}

Recall that using capped adaptive sampling requires identifying the set
\begin{equation*}
    \cW_k = \left\{i \; | \; f_i(x^k) \geq \theta \max_{j=1,\ldots, q} f_j(x^k) + (1-\theta)\EE{j\sim p}{f_j(x^k)}\right\}.
\end{equation*}
Sampling with the capped adaptive sampling strategy requires $q + O(\log q)$ flops to identify the maximal sketched loss $f_i(x^k)$,  $2q$ flops to computed the weighted average of the sketched losses $\EE{j\sim p}{f_j(x^k)}$, $O(1)$ flops to calculate the threshold for the set $\cW_k$, $q$ flops to apply the threshold to the sketched losses to determine the set $\cW_k$, and on average $2q$ flops to sample from the sketched losses contained in the set $\cW_k$ using inversion by sequential search. Thus, the total cost of the sampling step is $6q + O(\log q)$ flops. When a uniform average is used in place of the weighted average, the expected sketched loss $\EE{j\sim p}{f_j(x^k)}$ can be computed in just $q$ flops as opposed to $2q$. In that case, the total cost of the sampling step is only $5q + O(\log q)$.

\section{Auxiliary lemma}\label{sec:aux_lemmas}
We now invoke a lemma taken from~\cite{gower2016pseudo}.
\begin{lemma}\label{lem:NullA}For any matrix $\mW$ and symmetric positive semidefinite matrix $\mG$ such that
\begin{equation} \label{eq:Gnullassm}
 \kernel{\mG} \subset \kernel{\mW^\top },
 \end{equation} we have that
\begin{equation}\label{eq:8ys98hs}\kernel{\mW} = \kernel{\mW^\top \mG \mW}.
\end{equation}
\end{lemma}
\begin{proof} 
In order to establish \Cref{eq:8ys98hs}, it suffices to show the inclusion $\kernel{\mW} \supseteq \kernel{\mW^\top \mG \mW}$ since the reverse inclusion trivially holds. Letting $s\in \kernel{\mW^\top \mG \mW}$, we see that $\|\mG^{1/2}\mW s\|^2=0$, which implies $\mG^{1/2}\mW s=0$.
Consequently 
\[\mW s \in \kernel{\mG^{1/2}} = \kernel{\mG} \overset{\eqref{eq:Gnullassm} }{\subset}  \kernel{\mW^\top}.\] Thus $\mW s \in \kernel{\mW^\top} \cap \range{\mW}$ which are orthogonal complements which shows that $\mW s = 0.$
\end{proof}

\section*{Acknowledgements}
Needell, Molitor and Moorman are grateful to and were partially supported by NSF CAREER DMS $\#1348721$ and NSF BIGDATA DMS $\#1740325$. Moorman was also funded by NSF grant DGE $\#1829071$. Gower acknowledges the support by grants from
DIM Math Innov R\'egion Ile-de-France (ED574 - FMJH), reference ANR-11-LABX-0056-LMH, LabEx LMH.

\bibliographystyle{plain}

\end{document}